\documentclass[12pt]{article}

\usepackage[utf8]{inputenc}
\usepackage[T1]{fontenc}
\usepackage{fullpage}
\usepackage{amsthm, amssymb, amsmath}
\usepackage{mathtools}
\mathtoolsset{showonlyrefs=true}
\usepackage{graphicx}
\usepackage{subcaption}
\usepackage{color}
\usepackage{bm} 
\usepackage{tikz}
\usetikzlibrary{calc}

\graphicspath{{figures/}} 

\definecolor{celestialblue}{rgb}{0.29, 0.59, 0.82}
\definecolor{BlueMAP}{RGB}{37, 62, 119}

\usepackage[colorlinks,
            linkcolor=celestialblue,
            citecolor=celestialblue,
            urlcolor=celestialblue]{hyperref}

\usepackage[breakable]{tcolorbox}
\newtcolorbox{boxblue}[1]{breakable,colback=white,colframe=BlueMAP,fonttitle=\bfseries,title=#1,left=4pt,right=4pt,top=4pt,bottom=6pt}
\newtcolorbox{boxgreen}[1]{breakable,colback=white,colframe=red,fonttitle=\bfseries,title=#1,left=4pt,right=4pt,top=4pt,bottom=6pt}

\numberwithin{equation}{section} 

\newcommand{\cB}{\ensuremath{\mathcal{B}}}
\newcommand{\cC}{\ensuremath{\mathcal{C}}}
\newcommand{\cD}{\ensuremath{\mathcal{D}}}

\newcommand{\cF}{\ensuremath{\mathcal{F}}}

\newcommand{\cH}{\ensuremath{\mathcal{H}}}

\newcommand{\cK}{\ensuremath{\mathcal{K}}}
\newcommand{\cL}{\ensuremath{\mathcal{L}}}
\newcommand{\cM}{\ensuremath{\mathcal{M}}}
\newcommand{\cN}{\ensuremath{\mathcal{N}}}
\newcommand{\cO}{\ensuremath{\mathcal{O}}}
\newcommand{\cP}{\ensuremath{\mathcal{P}}}

\newcommand{\cR}{\ensuremath{\mathcal{R}}}

\newcommand{\cV}{\ensuremath{\mathcal{V}}}
\newcommand{\cW}{\ensuremath{\mathcal{W}}}

\newcommand{\bB}{\ensuremath{\mathbb{B}}}

\newcommand{\bE}{\ensuremath{\mathbb{E}}}

\newcommand{\bG}{\ensuremath{\mathbb{G}}}

\newcommand{\bM}{\ensuremath{\mathbb{M}}}
\newcommand{\bN}{\ensuremath{\mathbb{N}}}

\newcommand{\bP}{\ensuremath{\mathbb{P}}}

\newcommand{\bR}{\ensuremath{\mathbb{R}}}

\newcommand{\bU}{\ensuremath{\mathbb{U}}}

\newcommand{\by}{\ensuremath{\textbf{y}}}

\newcommand{\rY}{\ensuremath{\mathrm{Y}}}

\def\[{\left[}
\def\]{\right]}
\def\<{\langle}
\def\>{\rangle}
\def\({\left(}
\def\){\right)}
\def\[{\left [}
\def\]{\right]}
\def\({\left(}
\def\){\right)}
\def\wt{\widetilde}

\newcommand{\bb}{{\bf b}}
\def\bw{{\bf w}}
\def\bu{{\bf u}}
\def\bv{{\bf v}}
\newcommand{\bx}{{\bf x}}

\newcommand{\bbR}{{\bf R}}
\newcommand{\bbQ}{{\bf Q}}
\def\vp{\varphi}

\def\Chi{\raise .3ex\hbox{\large $\chi$}}

\newcommand{\prox}{\mathrm{prox}}
\newcommand{\epi}{\text{epi}}
\newcommand{\pa}[1]{\left(#1\right)}

\newcommand{\norm}[1]{\Vert #1 \Vert}

\newcommand{\Vn}{\ensuremath{{V_n}}}
\newcommand{\Wm}{\ensuremath{{W_m}}}

\newcommand{\cond}{\; :\;}

\newcommand{\wc}{\ensuremath{\mathrm{wc}}}
\newcommand{\ms}{\ensuremath{\mathrm{ms}}}

\newcommand{\wca}{\ensuremath{\mathrm{wca}}}

\newcommand{\rad}{\ensuremath{\mathrm{rad}}}

\newcommand{\joint}{\ensuremath{\text{joint}}}
\newcommand{\aff}{\ensuremath{\text{aff}}}

\newcommand{\dist}{\operatorname{dist}}

\newcommand{\opt}{\ensuremath{\mathrm{opt}}}
\DeclareMathOperator*{\argmax}{arg\,max}
\DeclareMathOperator*{\argmin}{arg\,min}

\newcommand{\dx}{\ensuremath{\mathrm dx}}
\newcommand{\dy}{\ensuremath{\mathrm dy}}

\DeclareMathOperator*{\vspan}{span}

\newcommand{\eps}{\ensuremath{\varepsilon}}
\newcommand{\e}{\ensuremath{\varepsilon}}
\newcommand{\cen}{\ensuremath{\mathrm{cen}}}

\newbool{journalTemplate} 
\boolfalse{journalTemplate}
\ifbool{journalTemplate}{
}
{

\newtheorem{theorem}{Theorem}[section]
\newtheorem{lemma}[theorem]{Lemma}

\newtheorem{corollary}[theorem]{Corollary}

\newtheorem{remark}[theorem]{Remark}

\theoremstyle{definition}
\newtheorem{definition}{Definition}
}

\begin{document}

\title{Inverse Problems: A Deterministic Approach \\using Physics-Based Reduced Models
}

\author{Olga Mula}

\maketitle

\abstract{
These lecture notes summarize various summer schools that I have given on the topic of solving inverse problems (state and parameter estimation) by combining optimally measurement observations and parametrized PDE models. After defining a notion of optimal performance in terms of the smallest reconstruction error that any reconstruction algorithm can achieve, the notes present practical numerical algorithms based on nonlinear reduced models for which one can prove that they can deliver a performance close to optimal. We also discuss algorithms for sensor placement with the approach. The proposed concepts may be viewed as exploring alternatives to Bayesian inversion in favor of more deterministic notions of accuracy quantification.
}

\tableofcontents

\section{Introduction}

Inverse problems aim to find the causal factors that lead to a set of observed effects. As the term itself indicates, they are the inverse of direct or forward problems, which start with the causes and then calculate the effects. By their very nature, both forward and inverse  problems are ubiquitous in science and engineering. Let us start with a few examples. Suppose we are interested in the sun's surface temperature. One way of accessing to this information is by measuring the amount of light that the sun emits at each wavelength. The forward problem here consists of describing spectral radiance as a function of wavelength and temperature. The most common physical model giving this relation is Planck's law of black body radiation. The inverse problem is the one of estimating the temperature from the observed spectral radiance. As another example, we could mention the famous inverse problem of ``hearing the shape of a drum''. This question can be traced back at least to the works of Weyl in the early 1910s (see \cite{Weyl1911}), and has motivated important advances in spectral theory. The idea is that the frequencies at which a drumhead can vibrate depend on its shape. The forward problem here is to develop a physical model relating a given shape to the acoustic frequencies. This is described by the Helmholtz equation, and the acoustic frequencies are the eigenvalues of a Laplacian in space. A central inverse problem is whether the shape can be predicted if the frequencies (namely the eigenvalues of the operator) are known.  In the early 1990s, it was proven that different shapes can yield the same acoustic frequencies, thus answering negatively to the question as to whether one can hear the shape of a drum (see \cite{GWW1992}).

The above examples illustrate the main properties of inverse problems. They are typically \emph{ill-posed} in the sense that they do not necessarily have unique solutions. They are \emph{unstable}: deviations in the observed input data caused by measurement noise can cause arbitrarily large perturbations in the results. They are also \emph{nonlocal}: in the example about the drum, the observed frequencies depend on the propagation of sound waves everywhere on the drum's surface and their reflections at the border. In time-dependent phenomena, inverse problems are also \emph{noncausal}: if we try to estimate the initial temperature distribution in a room based on the observed temperature at some points at the final time, we find that vastly different initial conditions may have produced the final condition, at least within the accuracy limit of our measurements.

Several different approaches exist to solve inverse problems. Their common denominator is that they all incorporate additional a priori information in order to fight against ill-posedness, and derive useful reconstructions. Note that we have already tacitly introduced a priori hypotheses in the above examples by assuming that the studied phenomena can be well  described by certain physical models. These models usually come in the form of ordinary or partial differential equations. Taking them as priors gives raise to a large family of strategies aiming to blend complex physical models with often vast data sets which are now routinely available in many applications. The Bayesian approach is probably the most widespread technique belonging to this family (see, e.g., \cite{Stuart2010, DS2017}). One models available a priori information as a probability distribution (the prior), and uses the measurement data to compute a posterior distribution that represents the uncertainty in the solution. The approach has the appealing property of providing a quantification of uncertainty in the reconstructions. However, since it is based on sampling the posterior distribution, it quickly suffers from a high numerical cost, especially in a high dimensional framework.

In these lecture notes we present an alternative approach to Bayesian inversion which allows to provide a more deterministic accuracy quantification of the outputs. Instead of formulating the priors as probability distributions, one only assumes that a certain parametrized PDE is a good physical model for the system under consideration. Taking this point of view, we discuss optimality criteria which define intrinsic limits regarding the best possible reconstruction accuracy that one can achieve when solving an inverse problem. Reduced Ordel Models play a significant role in this approach since they can be used to build efficient computational strategies whose performance approaches optimality.

It is important to note that physical models are actually not the only prior assumption that one can make to solve inverse problems. One can alternatively resort to regularization methods which impose certain smoothness requirements or closeness to certain reference functions in order to build robust (pseudo-) inverses to the ill-posed inverse problems. We refer to \cite{EHN1996} for linear regularization methods and to \cite{BB2018} for an overview of state of the art nonlinear ones. This approach is particularly appealing when there is no clear description of the problems in terms of a physical model. Certain applications related to imaging such as tomography inversion are particularly well suited for this approach.

These lecture notes are based mainly on the results of \cite{MMPY2015, MPPY2015, BCDDPW2017, BCMN2018, CDDFMN2020, CDMN2020}. We will outline connections to other contributions and other topics along the way.  

\section{Forward and Inverse Problems}
\label{sec:intro-fwd-inv}

Parametrized partial differential equations play a central role in the approach that we present. They are of common use to model complex physical systems, and are routinely involved in design and decision-making processes. Such equations can generally be written in abstract form as
\begin{equation}
\cP(u,y)=0,
\label{eq:genpar}
\end{equation}
where $\cP$ is a partial differential operator, and $y=(y_1,\dots,y_p)$ is a vector of
scalar parameters ranging in some domain $\rY\subset \bR^p$.
We assume well-posedness, that is, for any $y\in Y$  the problem admits a unique solution $u=u(y)$ in some 
Hilbert space $V$ whose elements depend on a physical variable $x$ ranging in a domain $\Omega\subset \bR^d$.
The variable $x$ usually refers to space but it is not limited to that meaning, and it may also refer to more elaborate sets of variables such as space, time, momentum, and possibly others. We may thus regard $u$ as a function $(x,y)\mapsto u(x, y)$ from $\Omega\times \rY$ to $\bR$,
or we may also consider the \emph{parameter to solution map}
\begin{equation}
y\mapsto u(y),
\label{eq:solmap}
\end{equation}
from $\rY$ to $V$. This map is typically nonlinear, as well as the {\it solution manifold}
\begin{equation}
\cM:=\{u(y) \, : \, y\in Y\}\subset V
\label{eq:manifold}
\end{equation}
which describes the collection of all admissible solutions. Throughout this document,
we assume that $\rY$ is compact in $\bR^d$ and that the map \eqref{eq:solmap}
is continuous. Therefore $\cM$ is a compact set of $V$. 
We sometimes refer to the solution $u(y)$ as the \emph{state} 
of the system for the given parameter vector $y$.

The parameters $y$ are used to represent physical quantities such as 
diffusivity, viscosity, velocity, source terms, or the geometry of the physical domain
in which the PDE is posed.  In several relevant instances, $y$ may be high or even
countably infinite dimensional, that is, $p\gg1$ or $p=\infty$.

Given this general setting, two families of problems may be considered:
\begin{enumerate}
\item \emph{Forward problems} are concerned with the parameter to solution map \eqref{eq:solmap}. For a given parameter $y\in \rY$, the goal is to develop numerical schemes to solve the PDE problem \eqref{eq:genpar}. This is an old topic with a long history in numerical analysis. It can be addressed with classical discretization techniques such as finite element, finite volume spectral methods, or, less classically, with machine learning techniques such as, for example, Physics-Informed Neural Networks. For general references to these methods, we refer to \cite{EG2013, Leveque2002, RPK2019, BM1997, CHQ2012}.

In numerous design and decision-making processes, one is often confronted to optimization problems defined over the solution manifold $\cM$. The algorithms for this task are usually iterative and require to evaluate many solutions $u(y)$ on a large set of dynamically updated parameters $y\in\rY$. Computations cannot be addressed rapidly unless the overall complexity has been appropriately reduced, and motivates the search for accurate methods to approximate the family of solutions very quickly at a reduced computational cost. This task, usually known as \emph{reduced modelling}, \emph{model order reduction}, or \emph{metamodeling}, has classically been addressed by approximating $\cM$ with well-chosen linear subspaces of $V$. However, it can be expected to be successful only when the Kolmogorov $n$-width of $\cM$ decays fast with $n$. For a given $n\geq 1$, this quantity is defined as
\begin{equation}
\label{eq:kolmo}
d_n(\cM)\coloneqq \inf_{ \substack{V_n \subset V \\ \dim(V_n)\leq n}} \sup_{u\in \cM} \inf_{v\in V_n} \Vert u - v \Vert,
\end{equation}
and it quantifies the best approximation of $\cM$ that one can achieve when using linear subspaces of $V$ of dimension lower or equal to $n$. While $d_n(\cM)$ decays quickly for certain families of parabolic or elliptic problems (see \cite{CD2015}), most transport-dominated problems are expected to present a slow decaying width and require to study nonlinear approximation methods. This is a field of very active study which we will not cover in these notes but it is tightly related to some computational issues for solving inverse problems with transport phenomena that we outline later on.

\item \emph{Inverse Problems} occur when the parameter $y$ is not given, and, instead, we only observe a vector of {\it linear} measurements
\begin{equation}
z_i=\ell_i(u),\quad i=1,\dots,m,
\end{equation}
where each $\ell_i\in V'$ is a known continuous linear functional on $V$. 
The $\ell_i$ are a mathematical
model for sensors that capture some partial information on the 
unknown solution $u(y)\in V$. We will also sometimes use notation in terms of the vector of observations
\begin{equation}
z=(z_1,\dots,z_m)^T=\ell(u) \in \bR^m, \quad \ell=(\ell_1,\dots,\ell_m).
\end{equation}
In this setting, the goal is to recover the unknown state $u\in \cM$ from $z$ or even the underlying parameter vector $y\in Y$ for which $u=u(y)$.
Therefore, in an idealized setting, one observes the result of the composition map
\begin{equation}
\label{eq:fwd-maps}
y\in Y \mapsto u \in \cM \mapsto z\in \bR^m.
\end{equation}
for the unknown $y$. In inverse problems, the goal is to ``revert the sense of the arrows'' in the above cascade of forward mappings \eqref{eq:fwd-maps}. This leads to two main types of inverse problems:
\begin{enumerate}
\item
\emph{State estimation:} recover an approximation $u^*$ of {the state} $u$ from the observation {$z=\ell(u)$} and assuming that $u$ belongs to the manifold $\cM$. This inverse problem is linear in nature because the forward map $\ell:u\mapsto \ell(u)=z$ is linear. It is however challenging because the target $u$ lives in $V$, which is a space of typically very high or infinite dimension. In addition, the information that $u$ belongs to $\cM$ is difficult to handle given that $\cM$ has a complicated geometry, which is only partially known to us by solving forward problems $y\mapsto u(y)$ for different values of $y\in \rY$.
\item
\emph{Parameter estimation: }recover an approximation $y^*$ of the parameter $y$ from the observation {$z = \ell(u)$ when $u=u(y)$}. This is a nonlinear inverse problem, for which the prior information available on $y$ is given by the domain $Y$.
\end{enumerate}
\end{enumerate}

Note that so far we have carried the discussion in a very idealized setting since we have assumed that:
\begin{itemize}
\item the modeling of the sensor response through the $\ell_i$ is perfect,
\item there is no observation noise, and, even if we had noise, we would need to suppose a certain model for it,
\item the PDE model perfectly describes reality, that is, there exists a parameter $y\in \rY$ such that $u=u(y)$ for the observations $\ell(u)$.
\end{itemize}
Of course, none of these modeling assumptions are satisfied in reality, and it is important to estimate their impact. However, they add an extra layer of complexity in the mathematical analysis of optimal reconstruction benchmarks that we are interested in. We thus proceed in two steps: we first place ourselves in the idealized setting without modeling errors, and analyze optimality benchmarks related to intrinsic limits regarding the best possible reconstruction accuracy. We then extend the analysis to account for modeling errors.

\section{Optimality Benchmarks for State Estimation}
\label{sec:optim-benchmarks}

Let us place ourselves in the idealized setting without modeling errors, and let us consider the state estimation problem of approximating an unknown function $u\in V$ from data given by $m$ linear measurements
\begin{equation}
z_i = \ell_i(u), \quad i=1,\dots,m,
\end{equation}
where the $\ell_i$ are $m$ linearly independent linear functionals over $V$.

Denoting by $\omega_i\in V$ the Riesz
representers of the $\ell_i$, such that $\ell_i(v)=\<\omega_i,v\>$ for all $v\in V$, and defining the observation space
\begin{equation}
W:={\rm span}\{\omega_1,\dots,\omega_m\},
\end{equation}
the measurement data $z = (z_1,\dots, z_m)^T$ are equivalently represented by
\begin{equation}
\omega=P_{W}u,
\end{equation}
where $P_W$ is the orthogonal projection from $V$ onto $W$. This equivalence comes from the fact that we can write
$$
P_W u = \sum_{i=1}^n c_i \omega_i,
$$
for some coefficients $c_i\in \bR$. Knowing the measurement data allows us to write that
$$
z_i = \ell_i(u) = \left< \omega_i, u\right> = \left< \omega_i, P_W u\right> = \sum_{j=1}^m c_j \left< \omega_i, \omega_j \right>,\quad \forall i=1,\dots, m.
$$
Thus the vector of coefficients $c =(c_1,\dots,c_m)^T$ is the unique solution to the linear system 
$$
\bB \,c = z
$$
where
$$
\bB \coloneqq ( \left< \omega_i, \omega_j \right>)_{1\leq i, j \leq m} 
$$
is an invertible matrix because the $\{ \omega_i \}_{i=1}^m$ are linearly independent (because their associated linear functionals $\{ \ell_i \}_{i=1}^m$ are assumed to be  independent). Therefore knowing $z$ is equivalent to knowing $c$, and also $w =P_W u$.

A {\it recovery algorithm} is a map
\begin{equation}
A: W \to V
\end{equation}
and the approximation to $u$ obtained by this algorithm is
\begin{equation}
u^*=A(w)=A(P_Wu).
\end{equation}
Note that, in our terminology, an algorithm can be computationally feasible or not. At this stage, it is just a mapping from the observation space $W$ to the ambient space $V$, and we do not attach any notion of practical feasibility to it. We will add this idea in a second stage.

The construction of $A$ should be based on the available prior 
information that describes the properties of the unknown $u$, and the evaluation of its performance needs
to be defined in some precise sense. Two distinct avenues can be followed: 
\begin{itemize}
\item
In the {\it deterministic setting}, the sole prior information 
is that $u$ belongs to the solution manifold $\cM$ that we defined in equation \eqref{eq:manifold}. The performance of an algorithm $A$ over the class $\cM$ is measured by the ``worst case'' reconstruction error
\begin{equation}
\label{eq:err-wc}
E_{\wc}(A,\cM)=\sup\{\|u-A(P_Wu)\|\; : \; u\in\cM\}.
\end{equation}
The problem of finding an algorithm $A$ that minimizes $E_{\wc}(A, \cM)$ is 
called {\it optimal recovery}. It has been extensively studied for convex sets $\cM$ that are 
balls of smoothness classes \cite{Bojanov1994,MR1977,NW2008} but note that this is not the  present case for our solution manifold.
\item
In the {\it stochastic setting}, the prior information on $u$ is 
described by a probability distribution $p$ on $V$,
which is supported on $\cM$, typically induced by a probability distribution on $Y$ that is assumed
to be known.
It is then natural to measure the performance of an algorithm
in an averaged sense, for example through the mean-square error
\begin{equation}
E_{\ms}(A,p)=\bE(\|u-A(P_Wu)\|^2)=\int_V \|u-A(P_W u)\|^2 dp(u).
\label{eq:err-ms}
\end{equation}
This stochastic setting is the starting point { for} {\it Bayesian estimation} 
methods \cite{DS2017}. Let us observe that for any algorithm $A$ one has $E_{\ms}(A,p)\leq E_{\wc}(A,\cM)^2$.
\end{itemize}

In the following, we concentrate on the deterministic setting according to the
above distinction. In this setting, the performance benchmark of recovery algorithms is given by
\begin{equation}
\label{eq:opt-alg}
E^*_{\wc}(\cM)=\inf_{A:W\to V}E_{\wc}(A,\cM),
\end{equation}
where the infimum is taken over all possible maps $A:W\to V$.

In \cite{BCDDPW2017, CDDFMN2020}, the authors give a simple mathematical description of an optimal map that meets this benchmark. To define it, we note that in the absence of model bias and when a noiseless measurement $w=P_W u$ is given, our knowledge on $u$ is that it belongs to the set
\begin{equation}
\cM_w:=\cM\cap (\omega+W^\perp).
\label{eq:M-omega}
\end{equation}
We refer to Figure \ref{fig:bird} fo a graphical illustration of $\cM_\omega$ and the ideas that we are about to introduce. The figure helps to see that $\cM_\omega$ can be understood as the ``slice'' of the manifold $\cM$ which agrees with a given observation $\omega\in W$. This ``slice'' can be a fully connected set, or composed of non connected sets, and it could even be the empty set depending on $\omega$.
\begin{figure}[ht]
  \centering
  \includegraphics[scale=0.4]{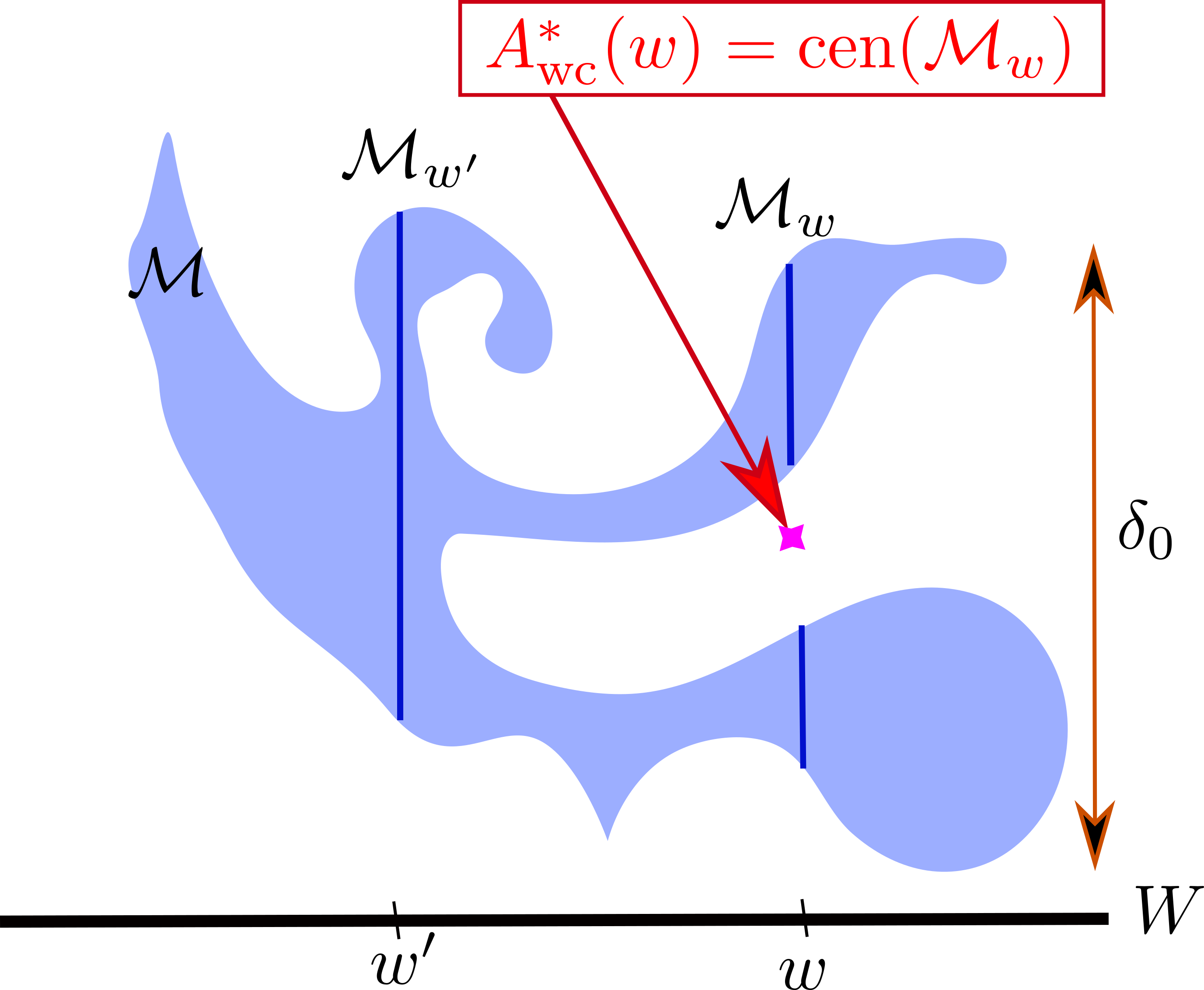}
\caption{Illustration of the optimal recovery benchmark on a manifold in the two dimensional Euclidean space. Note that sometimes $\cM_\omega$ may be a non connected set as the figure depicts. This does not alter the fact that the best reconstruction is the center of the Chebyshev ball $\cen(\cM_\omega)$. Remark that the center does not necessarily lie in $\cM_\omega$ when the set is non connected.}
\label{fig:bird}
\end{figure}

The best possible recovery map can be described through the following general notion.

\begin{definition}
The Chebychev ball of a bounded set $S\in V$ is the closed ball $B(v,r)$ of minimal radius 
that contains $S$. One denotes by $v={\rm cen}(S)$ the Chebychev center of $S$ and 
$r={\rm rad}(S)$ its Chebychev radius. 
\end{definition}

In particular 
one has
\begin{equation}
\frac 1 2 {\rm diam}(S)\leq {\rm rad}(S)\leq  {\rm diam}(S),
\label{radiam}
\end{equation}
where ${\rm diam}(S):=\sup\{\|u-v\|\,:\, u,v\in S\}$ is the diameter of $S$.
Therefore, the recovery map that minimizes the worst case error over $\cM_w$ for
any given $w$, and therefore over $\cM$ is
defined by
\begin{equation}
\label{eq:optimal-A-Cheb-center}
A_\wc^*(w)={\rm cen}(\cM_w).
\end{equation}
Its associated worst case error is
\begin{equation}
E_\wc^*(\cM)= \sup \{\rad(\cM_w)\, :\, w\in W\}.
\end{equation}
Note that the map $A^*_\wc$ is also optimal among all algorithms for each slice $\cM_w$, where $w\in P_W (\cM)$, since 
\begin{equation}
E_\wc(A^*_\wc,\cM_w)=\min_{A}E_\wc(A,\cM_w)=\rad(\cM_w), \quad \forall w\in P_W(\cM).
\end{equation}
However, there may exist other maps
$A$ such that $E_\wc(A,\cM)=E_\wc^*(\cM)$, since we 
also supremize over $w\in P_W(\cM)$.

In view of the equivalence \eqref{radiam}, we can
relate $E^*_\wc(\cM)$ to the quantity 
\begin{equation}
\label{eq:benchmark-M}
\delta_0=\delta_0(\cM,W):=\sup\{{\rm diam}(\cM_w)\,:\, w\in W\}=\sup \{\|u-v\|\; : \; u,v\in \cM, \;u-v\in W^\bot \},
\end{equation}
by the equivalence
\begin{equation}
\label{eq:equiv-optimality}
\frac 1 2 \delta_0\leq E^*_\wc(\cM) \leq \delta_0.
\end{equation}

Note that injectivity of the measurement map $P_W$ over $\cM$
is equivalent to $\delta_0=0$. More importantly, note that, in practice, the above map $A^*_\wc$ cannot be easily constructed. Since the solution manifold $\cM$ is a high-dimensional and geometrically complex object, one cannot easily find the Chebyshev center to $\cM_\omega$ for a given measurement $\omega$. One is therefore interested in designing ``sub-optimal yet good'' recovery algorithms and analyze their performance. We discuss several possible approaches in Sections \ref{sec:optimal-affine}
to \ref{sec:piecewise-affine}.



\section{Optimal Affine Algorithms}
\label{sec:optimal-affine}

\subsection{Definition and preliminary remarks}
One possibility to find easily computable surrogates for the optimal map $A^*_\wc$ is to restrict the search to linear recovery mappings $A\in \cL(W,V)$. As we are going to see, finding good linear recovery maps is connected to finding good linear subspaces to approximate the solution manifold $\cM$. The task is thus connected to {\it  reduced modeling} but we will see that there is a distinction to be made between the linear subspaces that one should use for forward problems and for inverse problems.

Generally speaking, \emph{forward} reduced modeling consists of building linear spaces $(V_n)_{n\geq 0}$ with
increasing dimension $\dim(V_n)=n$ which uniformly approximate
the solution manifold in the sense that
\begin{equation}
\dist(\cM,V_n)\coloneqq \max_{u\in \cM}\|u-P_{V_n}u\|\leq \e_n,
\label{eq:epsn}
\end{equation}
where 
\begin{equation}
\e_0\geq \e_1\geq \cdots \geq \e_n\geq \cdots \geq 0,
\end{equation}
are known tolerances. Instances of reduced models for parametrized families of PDEs
with provable accuracy are provided by polynomial approximations in the $y$ variable
\cite{CD2015acta,CDS2011} or reduced bases \cite{BMPPT2012,RHP2007}. The construction of a reduced model
is typically done offline, using a large training set of instances of $u\in\cM$
called {\it snapshots}. The offline stage potentially has a high
computational cost. Once this is done, the online cost of recovering
$u^*=A(w)$ from any data $w$ using this reduced model should 
in contrast be moderate. 

In \cite{MPPY2015}, a simple reduced-model based recovery algorithm was proposed. Assuming that we have a reduced model $V_n$, the algorithm, called Parametrized Background Data-Weak (PBDW), is defined in terms of the map
\begin{equation}
A_n(w):={\rm argmin}\{ \dist(v,V_n) \; :\; v\in \omega+W^\perp\},
\label{eq:pbdw-lin}
\end{equation}
which is well defined provided that $V_n\cap W^\perp =\{0\}$.  A necessary (but not sufficient) condition to guarantee well-posedness is to have $n\leq m$, which we will assume in the following.


\begin{figure}
     \centering
     \begin{subfigure}[b]{0.4\textwidth}
         \centering
	\includegraphics[width=\textwidth]{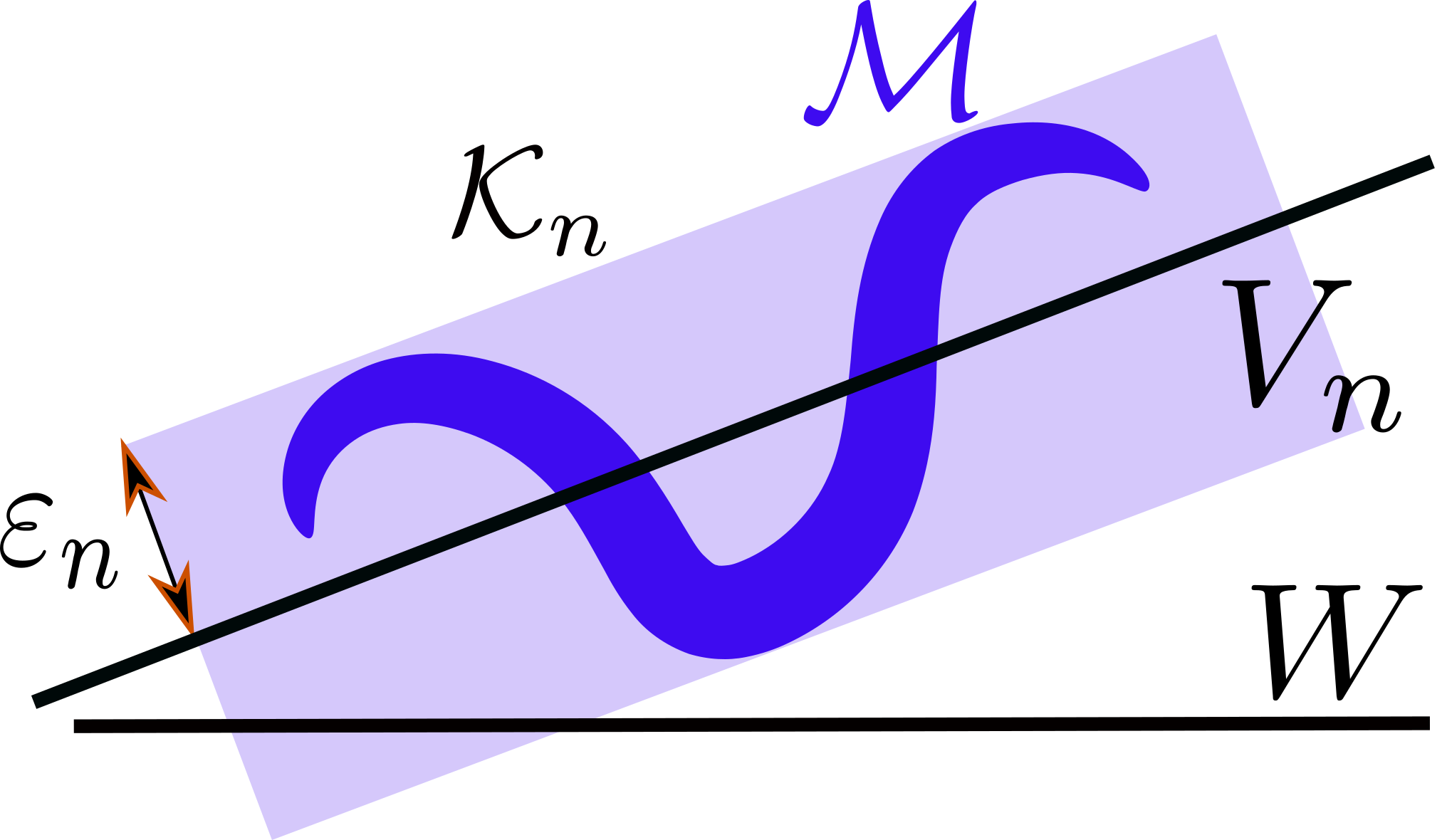}
         \caption{A manifold $\cM$, a linear space $V_n$ with accuracy $\e_n$, and the cylinder $\cK_n$.}
         \label{fig:Kn}
     \end{subfigure}
     \begin{subfigure}[b]{0.4\textwidth}
         \centering
         \includegraphics[width=\textwidth]{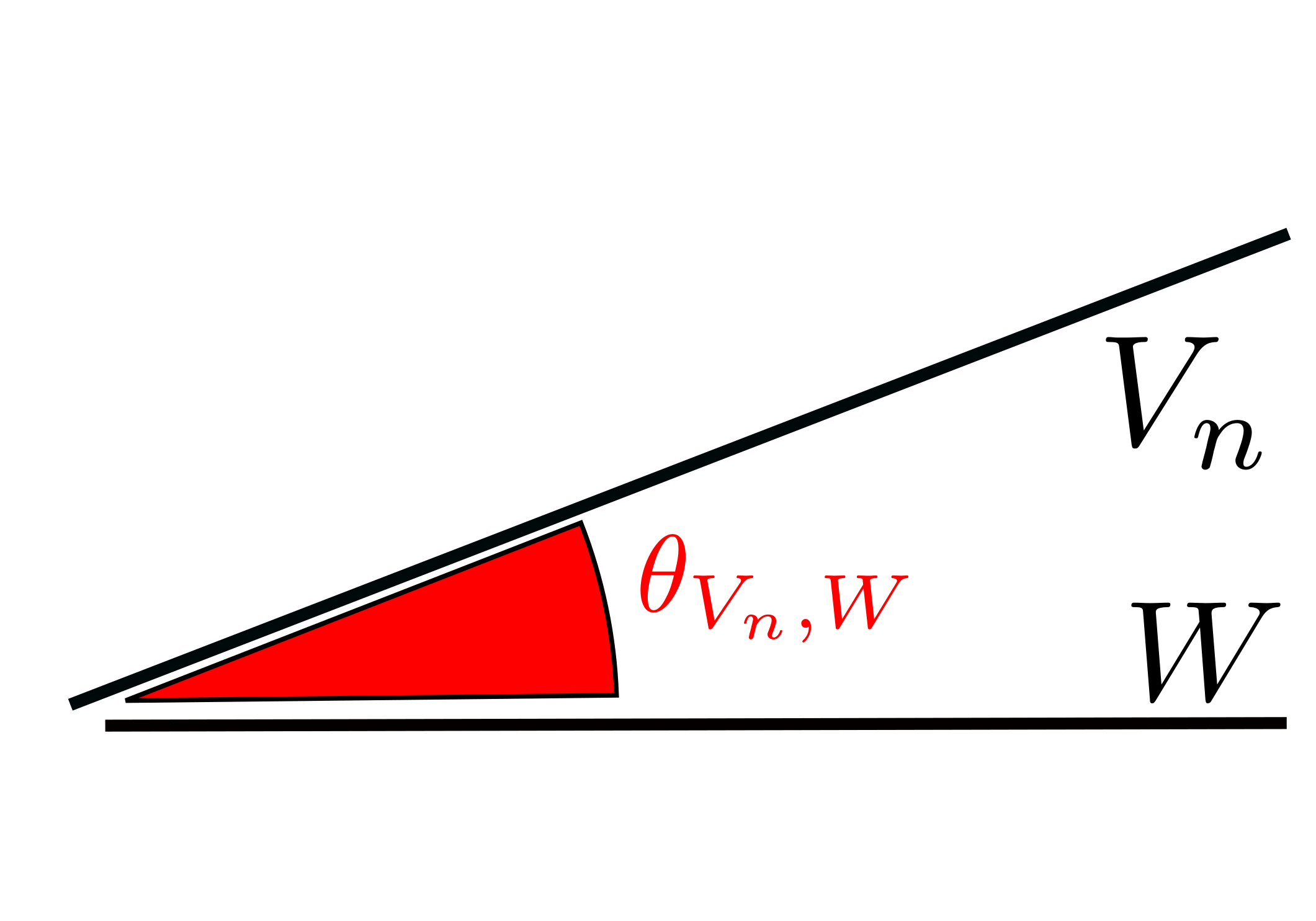}
         \caption{Angle between $V_n$ and $W$.}
         \label{fig:angle}
     \end{subfigure}
        \caption{The concepts associated to the linear reconstruction algorithm.}
        \label{fig:one-space}
\end{figure}

We can prove that $A_n$ is a linear mapping in $\cL(W, V)$ and it was shown in \cite{BCDDPW2017} that $A_n$ has a simple interpretation in terms of the cylinder (see Figure \ref{fig:Kn})
\begin{equation}
\label{eq:cylinder}
\cK_n:=\{ v\in V \; : \: {\rm dist}(v,V_n)\leq \e_n\},
\end{equation}
that contains the solution manifold $\cM$. Namely, the algorithm $A_n$ is also given by
\begin{equation}
A_n(w)=\cen(\cK_{n,w}), \quad \cK_{n,w}:= \cK_n\cap (\omega+W^\perp),
\end{equation}
and the map is shown to be optimal among all linear and nonlinear algorithms when $\cM$ is replaced by the simpler containment set $\cK_n$, that is
$$
A_n = \argmin_{A:W\to V} E_{\wc}(A, \cK_n). 
$$
The substantial advantage of this approach is that, in contrast to $A^*_\wc$, the map $A_n$ can be easily computed by solving a simple least-squares minimization problem of size $n\times m$. Appendix \ref{appendix:linear-pbdw} explains how to compute $A_n$ in practice. Note that $A_n$ depends on $V_n$ and $W$, but not on $\e_n$ in view of \eqref{eq:pbdw-lin}. This is important because $\e_n$ is only known approximately in practice.

This algorithm satisfies the performance bound (see \cite{MPPY2015, BCDDPW2017})
\begin{equation}
\label{eq:pbdw-err-bound}
\|u-A_n(P_Wu)\|\leq \mu_n {\rm dist}(u,V_n\oplus (V_n^\perp \cap W))
\leq \mu_n {\rm dist}(u,V_n) \leq \mu_n\e_n,
\end{equation}
where the last inequality holds when $u\in \cM$. Here 
\begin{equation}
\mu_n=\mu(V_n,W):=\max_{v\in V_n} \frac{\|v\|}{\|P_W v\|}
\end{equation}
is the inverse of the inf-sup constant
\begin{equation}
\label{eq:beta}
\beta_n:=\min_{v\in V_n} \max_{w\in W} \frac{\<v,w\>}{\|v\| \,\|w\|} = \min_{v\in V_n} \frac{\Vert P_W v \Vert}{\Vert v \Vert}.
\end{equation}
The quantity $\beta_n$ can be interpreted as the cosine of the angle between $V_n$ and $W$. This idea is illustrated in Figure \ref{fig:angle} through the angle denotes as $\theta_{V_n, W}$. In particular, note that $\beta_n =0$ (thus $\mu_n=\infty$) in the event where $V_n\cap W^\perp$ is non-trivial. Appendix \ref{appendix:beta} explains how to compute $\beta_n$ in practice.

An important observation for what is presented in what follows is that the PBDW algorithm 
\eqref{eq:pbdw-lin} has a simple extension to the setting where  $V_n$ is an affine space  rather than a linear space, namely,  when 
\begin{equation}
V_n^{(\aff)}=\bar u + V_n,
\label{affine}
\end{equation}
with $V_n$ a linear subspace of dimension $n$ and $\bar u$ a given offset that is known to us. In this case, denoting
$$
\bar \omega \coloneqq P_{W} \bar u,
$$
the affine version of \eqref{eq:pbdw-lin} reads
\begin{equation}
A^{(\aff)}_n(w):=\argmin\{ \dist(v,\bar u + V_n) \; :\; v\in \omega+W^\perp\},
\label{eq:pbdw-aff}
\end{equation}
which can also be written as
\begin{equation}
A^{(\aff)}_n(w) = \bar u + A_n(\omega-\bar \omega).
\end{equation}
At first sight, affine spaces may not seem to bring any significant improvement in terms of approximating the solution manifold,
due to the following observation: if $\cM$ is approximated with accuracy $\e$ by an $n$-dimensional affine space $V_n$ given by \eqref{affine},
it is also approximated with accuracy $\wt \e\leq \e$ by the $n+1$-dimensional linear space
\begin{equation}
\wt V_{n+1}:=V_n\oplus \bR \bar u.
\end{equation}
However, the choice of an affine subspace may significantly improve the performance of the algorithm \eqref{eq:pbdw-lin} in the case where the parametric solution $u(y)$ is a ``small perturbation''
of a nominal solution $\bar u=u(\bar y)$ for some $\bar y\in Y$, in the sense that 
\begin{equation}
{\rm diam}(\cM)\ll \|u\|.
\end{equation}
Indeed, suppose in addition that $\bar u$ is badly aligned
with respect to the measurement space $W$ in the sense that 
\begin{equation}
\|P_W \bar u \| \ll \|u\|.
\end{equation}
In such a case, any linear
space $V_n$ that is well tailored to approximating the solution manifold 
(for example a reduced basis space) will contain a direction
close to that of $\bar u$ and thus, we will have that $\mu_n \gg 1$, rendering the
reconstruction by the linear PBDW method much less accurate than
the approximation error by $V_n$. The use of the affine mapping
\eqref{affine} has the advantage of elimitating the bad direction $\bar u$ since
$\mu_n$ will now be computed with respect to the linear part $V_n$.


\medskip

The above algorithms $A_n$ and $A_n^{(\aff)}$ are defined in general for any subspace $V_n$. So they work with standard constructions of reduced models. These constructions are  tailored for the forward problem, and they targeted at making the spaces $V_n$ as efficient as possible for approximating $\cM$, that is, making $\dist(\cM,V_n)$ as small as possible for each given $n$. This implies that $\e_n$ is also small. For example, for the reduced basis spaces, it is known \cite{BCDDPW2011,DPW2013} that a certain greedy selection of snapshots generates spaces $V_n$ such that $\dist(\cM,V_n)$ decays at the same rate (polynomial or exponential) as the Kolmogorov $n$-width $d_n(\cM)$ (see equation \eqref{eq:kolmo}). However these constructions do not ensure the control of $\mu_n$ and therefore these reduced spaces for forward modeling may be much less efficient when using the PBDW algorithm for the inverse recovery problem.

In view of these observations, two main strategies are possible.
First, we can build affine spaces $V_n$ that are better targeted towards the 
recovery task. In other words, we want to build spaces $V_n$ to make
the recovery algorithm $A_n$ as efficient as possible given the measurement space $W$.
In fact, it was shown in \cite{CDDFMN2020} that one can find the optimal affine subspace for the state estimation problem. We summarize the main results on this front in Sections \ref{sec:optimal-affine-characterization} and \ref{sec:practical-algo-opt-aff}. A second strategy can be considered if we are allowed to select the measurement functionals $\ell_i$ from some admissible dictionary. This amounts to fixing $V_n$, and optimizing the over the space $W$. We present some strategies for sensor placement in Section \ref{sec:sensor-placement}. Section \ref{sec:GEIM} summarizes strategies to make a join selection of $V_n$ and $W_m$. A particular algorithm for this approach is the Generalized Empirical Interpolation Method, which, in the present context, can be seen as a particular way of jointly selecting $V_n$ and $W_m$ when we impose $n=m$. Finally, in Section \ref{sec:piecewise-affine} we fix the observation space $W$, and we present a reconstruction strategy that goes beyond linear and affine algorithms based on piecewise affine reconstructions.

\subsection{Characterization of Affine Algorithms}
\label{sec:optimal-affine-characterization}
In \cite{CDDFMN2020}, the authors aim to characterize the best affine subspace $V_n$ to apply the PBDW algorithm \eqref{eq:pbdw-aff}, and to develop an implementable strategy to find it. Here, we consider our measurement system to be imposed on us, and therefore $W$ is fixed.

It turns out that searching for the best affine subspace $V_n$ for the PBDW algorithm \eqref{eq:pbdw-aff} is equivalent to searching for the best affine reconstruction map $A^*_\aff:W\to V$ defined as
\begin{equation}
\label{eq:opt-aff-alg}
A^*_\aff \;\in\; \underset{ {\substack{A:W\to V \\ A \text{ affine}}}}{\argmin}\; E_{\wc}(A,\cM),
\end{equation}
where the existence of the minimum is guaranteed under very mild assumptions as we outline next. Since $A^*_\aff $ reaches best  the performance among all affine algorithms, we can write
\begin{equation}
\label{eq:opt-aff-alg-benchmark}
E^*_{\wca}(\cM) \coloneqq \underset{ {\substack{A:W\to V \\ A \text{ affine}}}}{\min}\; E_{\wc}(A,\cM).
\end{equation}
Note that we wrote $E^*_{\wca}(\cM)$ with the subindex ``wca'' to indicate that it is the optimal performance in the \underline{w}orst \underline{c}ase sense among all \underline{a}ffine maps. Obviously, $E^*_{\wca}(\cM) \geq E^*_{\wc}(\cM)$ since $E^*_{\wc}(\cM)$ is the optimal performance in the \underline{w}orst \underline{c}ase among all maps (affine and nonlinear).

We next characterize $A^*_\aff$. In order to do this, as a first observation, note that since we are given the measurement observation $w$,
any algorithm $A$ which is a candidate to optimality must satisfy $P_W(A(w)) = w$
(otherwise the reconstruction error would not be minimized). Thus a necessary condition for optimality is that $A$ should have the form
\begin{equation}
\label{opform}
A(w)=w+B(w),
\end{equation}
where $B: W\to W^\perp$  with $W^\perp$ the orthogonal complement of $W$ in $V$.
Therefore, in going further, we always require that $A$ has the above form \eqref{opform} and concentrate on the construction of good affine lifting maps $B$.

Our next observation is that any affine algorithm  $A$ of the form \eqref{opform} can always be interpreted
as a PBDW algorithm  $A_n$ for a certain  space $V_n$ with $n\le m$. 
\begin{lemma}[See \cite{CDDFMN2020}]
\label{proplin} $A$ is an affine map of the form  \eqref{opform} if and only if there exists $\bar u\in V$ and a linear subspace $V_n$ of dimension $n\le m$ such that $A$ coincides with the affine PBDW algorithm \eqref{eq:pbdw-aff} for $V_n^{(\aff)}=\bar u + V_n$.
\end{lemma}
In view of this result, the search for an affine reduced model $\bar u + V_n$ that
is best tailored to the recovery problem is equivalent to the search of
an optimal affine map. The next result tells us that such an optimal map always 
exists when $\cM$ is a bounded set.

\begin{theorem}
Let $\cM$ be a bounded set. Then there exists a map $A^*_\wca$ that minimizes
$E_\wc(A,\cM)$ among all affine maps $A$.
\end{theorem}

\subsection{A practical algorithm for optimal affine recovery}
\label{sec:practical-algo-opt-aff}
\paragraph{Discretization and truncation:}
Since we are searching among algorithms of the form \eqref{opform}, we have that 
\begin{align*}
E_{\wc}(A^*_\aff ,\cM)
&= \min_{A:W\to V\, \text{affine}}\; \max_{u\in \cM} || u - A(\omega) ||  \\
&= \min_{B:W\to W^\perp\, \text{affine}}\; \max_{u\in \cM} || u - \omega - B(\omega) ||  \\
&= \min_{c\in W^\perp,\, B:W\to W^\perp\, \text{linear}}\; \max_{u\in \cM} || P_{W^\perp} u - c - B(\omega) ||.
\end{align*}
This means that the optimal affine recovery map is obtained by minimizing the
convex function
\begin{equation}
F(c,B)=\max_{u\in\cM} \|P_{W^\perp}u-c-B(P_W u)\|,
\end{equation}
over $W^\perp \times \cL(W,W^\perp)$. This optimization problem
cannot be solved exactly for two reasons:
\begin{enumerate}
\item
The sets $W^\perp$ as well as $\cL(W,W^\perp)$ are infinite dimensional
when $V$ is infinite dimensional.
\item
One single evaluation of $F(c,B)$ requires in principle to explore the entire manifold $\cM$.
\end{enumerate}

The first difficulty is solved by replacing $V$ by a subspace 
$Z_N$ of finite dimension $\dim(Z_N)=N$ that approximates the solution manifold $\cM$ with an accuracy
of smaller order than that expected for the recovery error. One possibility is
to use a finite element space $Z_N=V_h$ of sufficiently small mesh size $h$. However its resulting dimension $N=N(h)$
needed to reach the accuracy could still be quite large. An alternative is to
use reduced model spaces $Z_N$ which are more
efficient for the approximation of $\cM$.

We therefore minimize $F(c,B)$ over $\wt W^\perp \times \cL(W,\wt W^\perp)$,
where $\wt W^\perp$ is the orthogonal complement of $W$ in the space $W+Z_N$,
and obtain an affine map $\wt A_\wca$ defined by
\begin{equation}
\label{redopt}
\widetilde  A_\wca(w)=w+\bar c+\bar Bw, \quad (\bar c,\bar B):=\argmin\{\widetilde F(c,B): \ c\in \wt W^\perp, B\in\cL(W,\wt W^\perp)\}.
\end{equation}
with
\begin{equation}
\widetilde F(c,B)=\max_{u\in\widetilde \cM} \|P_{W^\perp}u-c-B(P_W u)\|.
\end{equation}
In order to compare the performance of $\widetilde  A_\wca(w)$ with that of $A^*_\wca$,
we first observe that
\begin{equation}
\label{eq:epsN}
\|P_{W^\perp}u-P_{\wt W^\perp}u\| \leq \e_N:=\sup_{u\in\cM}{\rm dist}(u,Z_N).
\end{equation}
For any $(c,B)\in W^\perp \times \cL(W,W^\perp)$, we define
$(\wt c,\wt B)\in \wt W^\perp \times \cL(W,\wt W^\perp)$ by $\wt c=P_{\wt W^\perp}c$
and $\wt B=P_{\wt W^\perp}\circ B$. Then, for any $u\in \cM$,
$$
\begin{array}{ll}
\|P_{W^\perp}u-\wt c-\wt B(P_W u)\| &\leq \|P_{\wt W^\perp}(P_{W^\perp}u-c-B(P_W u))\|
+\|P_{W^\perp}u-P_{\wt W^\perp}u\| \\
&\leq  \|P_{W^\perp}u-c-BP_W u\|+\e_N.
\end{array}
$$
It follows that we have the framing
\begin{equation}
E(A^*_\wca,\cM)\leq E(\wt A_\wca,\cM)\leq E(A_\wca^*,\cM)+\e_N,
\label{frame}
\end{equation}
which shows that the loss in the recovery error is at most of the order $\e_N$.

To understand how large $N$ should be, let us observe that a recovery map
$A$ of the form \eqref{opform} takes its values in the
linear space
\begin{equation}
F_{m+1}=\bR c \oplus {\rm range}(B),
\end{equation}
which has dimension $m+1$. It follows that
the recovery error is always larger than 
the approximation error by such a space. Therefore
\begin{equation}
E_{\wc}(A^*_\wca,\cM) \geq d_{m+1}(\cM),
\end{equation}
where $d_{m+1}(\cM)$ is the Kolmogorov $n$-width defined by \eqref{eq:kolmo} for $n=m+1$.
Therefore, if we could use the space $Z_n$ that exactly achieves
the infimum in \eqref{eq:kolmo}, we would be ensured that, with $N=m+1$, the 
additional error $\e_N=\delta_{m+1}(\cM)$ in \eqref{frame} is of smaller order than 
$E_{\wc}(A^*_\wca,\cM)$. As a result we would obtain the framing
\begin{equation}
E(A^*_\wca,\cM)\leq E(\wt A_\wca,\cM)\leq 2E(A_\wca^*,\cM).
\label{frame2}
\end{equation}
 In practice, since we do not have access to 
the $n$-width spaces, we use instead the reduced basis spaces $Z_n:=V_n$
which are expected to have comparable approximation performances in view of the results
from \cite{BCDDPW2011,DPW2013}.

The second difficulty is solved by replacing the set $\cM$ 
in the supremum that defines $F(c,B)$ by a discrete training set $\wt \cM$,
which corresponds to a discretization $\wt Y$ of the parameter domain $Y$, that is
\begin{equation}
\wt \cM:=\{u(y)\, : \, y\in\wt Y\},
\end{equation}
with finite cardinality.

We therefore minimize over $\wt W^\perp \times \cL(W,\wt W^\perp)$ the function
\begin{equation}
\wt F(c,B)=\sup_{u\in\wt\cM} \|P_{W^\perp}u-c-BP_W u\|,
\end{equation}
which is computable. The additional error resulting from this discretization can 
be controlled from the resolution of the discretization. Namely, let $\e>0$ be the smallest
value such that $\wt \cM$ is
an $\e$-approximation net of $\cM$, that is, $\cM$ is covered by the balls $B(u,\e)$
for $u\in \wt\cM$. Then, we find that
\begin{equation}
\wt F(c,B)\leq F(c,B)\leq \wt F(c,B)+\e \|B\|_{\cL(W,\wt W^\perp)},
\end{equation}
which shows that the additional recovery error will be of the order of $\e$
amplified by the norm of the linear part of the optimal recovery map.

One difficulty is that the cardinality of $\e$-approximation nets becomes potentially
untractable for small $\e$ as the parameter dimension becomes large, due to
the curse of dimensionality. This difficulty also occurs in forward problems when
constructing reduced basis by a greedy selection process which also needs
to be performed in sufficiently dense discretized sets. 
Recent results obtained in \cite{CDDN2020} show 
that, in certain relevant instances, $\e$-approximation nets can
be replaced by random training sets of smaller cardinality. One interesting
direction for further research is to apply similar ideas to the present context of inverse state estimation.

\paragraph{Optimization algorithms:}
\label{sec:optim}
As already brought up, the practical computation of $\wt A_{\wc}$ consists in solving
\begin{equation}
\label{eq:opt-hilbert}
\min_{(c,B) \in \wt W^\perp \times \cL(W,\wt W^\perp)} \underbrace{\max_{u\in\wt\cM} \|P_{W^\perp}u-c-BP_W u\|^2}_{= \widetilde F(c, B)},
\end{equation}
The numerical solution of this problem is challenging due to its lack of smoothness (the objective function $\widetilde F$ is convex but non differentiable) and its high dimensionality (for a given target accuracy $\e_N$, the cardinality of $\wt \cM$ might be large). One could use classical subgradient methods, which are simple to implement.
However these schemes only guarantee a very slow $O(k^{-1/2})$ convergence rate of the objective function, where $k$ is the number of iterations. As illustrated in \cite{CDDFMN2020}, this approach does not give satisfactory results: due to the slow convergence, the solution update of one iteration falls below machine precision before approaching the minimum close enough. This motivates the use of a primal-dual splitting method which is known to ensure a $O(1/k)$ convergence rate on the partial duality gap. We next briefly describe this method. 

We assume without loss of generality that $\dim(W+V_N) = m+N$ and that $\dim \wt W^\perp = N$. Let $\{\psi_i\}_{i=1}^{m+N}$ be an orthonormal basis of $W+V_N$ such that $W={\rm span}\{\psi_1,\dots,\psi_m\}$. Since for any $u\in V$,
$$
P_{W+V_N} u = \sum_{i=1}^{m+N} u_i \psi_i,
$$
the components of $u$ in $W$ can be given in terms of the vector $\bw= ( u_{i} )_{i=1}^{m}$ and the ones in $\wt W^\perp$ with $\bu = ( u_{i+m} )_{i=1}^{N}$. 

We now consider the finite training set
\begin{equation}
\wt\cM:=\{u^1,\dots,u^J\}, \quad J:=\#(\wt\cM)<\infty,
\end{equation}
and denote by $\bw^j$ and $\bu^{j}$ the vectors
associated to the snapshot functions $u^j$ for $j=1,\dots,J$. 
One may express the problem \eqref{eq:opt-hilbert} as the search for
\begin{equation}
\label{eq:opt-Rn}
\min_{ \substack{(\bbR, \bb) \in \\ \bR^{N\times m}\times \bR^{N}}} \max_{j=1,\dots,J} \Vert \bu^{j} - \bR \bw^{j} - \bb \Vert^2_{2}.
\end{equation}
Concatenating the matrix and vector variables $(\bbR, \bb)$ into a single 
$\bx \in \bR^{m(N+1)}$, we rewrite the above problem as
\begin{equation}
\label{eq:minP}
\min_{\bx \in \bR^{m(N+1)}} \max_{j=1,\dots,J} f_j(\bbQ_j\bx),
\end{equation}
where $\bbQ_j\in \bR^{N\times m(N+1)}$ is a sparse matrix built using the coefficients of $\bw^j$ and 
$f_j(\by) := \Vert \bu^{j} - \by \Vert^2_{2}$.

The key observation to build our algorithm is that problem~\eqref{eq:minP} can be equivalently written as a minimization problem on the epigraphs, i.e.,
\begin{equation}
\begin{aligned}
	&\min_{(\bx,t) \in \bR^{m (N+1)} \times \bR^+} t \quad \text{subject to} \quad f_j(\bbQ_j \bx) \leq t , \quad j=1,\dots,J \\
\iff&\min_{(\bx,t) \in \bR^{m (N+1)} \times \bR^+} t \quad \text{subject to} \quad (\bbQ_j \bx,t) \in \epi_{f_j}, \quad j=1,\dots,J,
\end{aligned}
\end{equation}
or, in a more compact (and implicit) form,
\begin{equation}\label{eq:minPepi}\tag{$\mathrm{P_{epi}}$}
\min_{(\bx,t) \in \bR^{m (N+1)} \times \bR^+} t + \sum_{j=1}^J \iota_{\epi_{f_j}}\pa{\bbQ_j\bx,t} .
\end{equation}
where, for any non-empty set $S$ the indicator function $\iota_S$ has value $0$ on $S$ and $+\infty$ on $S^c$.
 
This problem takes the following canonical expression, which is amenable to a primal-dual proximal splitting algorithm
\begin{equation}\label{eq:primal-problem}
\min_{(\bx,t) \in \bR^{m (N+1)} \times \bR} G(\bx,t) + F \circ L(\bx,t).
\end{equation}
Here, $G$ is the projection map for the second variable
\begin{equation}
G(\bx,t) = t,
\end{equation} 
the linear operator $L$ is defined by
\begin{equation}
L (\bx,t):=\pa{(\bbQ_1 \bx,t),(\bbQ_2 \bx,t),\cdots,(\bbQ_J \bx,t)}
\end{equation}
and acts from $\bR^{m (N+1)} \times \bR$ to $\times_{j=1}^J (\bR^N \times \bR)$
and the function $F$ acting from $\times_{j=1}^J (\bR^N \times \bR)$ to $\bR$ is defined by
\begin{equation}
F\( (\bv_1,t_1),\cdots,(\bv_J,t_J)\):= \sum_{j=1}^J \iota_{\epi_{f_j}}\pa{\bv_j,t_j}.
\end{equation}
Note that $F$ is the indicator function of the cartesian product of epigraphs.

Before introducing the primal-dual algorithm, some remarks are in order:
\begin{enumerate}
\item We recall that if $\phi$ is a proper closed convex function on $\bR^d$, its proximal mapping $\prox_\phi$ is defined by
\begin{equation}
\prox_\phi(y)={\rm argmin}_{\bR^d}\(\phi(x)+\frac 1 2\|x-y\|_2^2\).
\end{equation}
\item The adjoint operator $L^*$ is given by
\begin{equation}
\label{eq:Lstar}
L^*\( (\bv_1,t_1),\cdots,(\bv_J,t_J)\):= \pa{\sum_{j=1}^J \bbQ_j^T \bv_j,\sum_{j=1}^J t_j} .
\end{equation}
It can be easily shown that the operator norm of $L$ satisfies $\norm{L}^2 \leq J + \sum_{j=1}^J \norm{\bbQ_j}^2$.
\item Both $G$ and $F$ are simple functions in the sense that their proximal mappings, $\prox_{G}$ and $\prox_{F}$, can be computed in closed form.
\end{enumerate}

The iterations of the primal-dual splitting method read for $k\geq 0$,
\begin{equation}
\label{eq:its-pd}
\begin{aligned}
(\bx,t)^{k+1} &= \prox_{\gamma_G G} \( (\bx, t)^k - \gamma_G L^* \( \( (\bv_1,\xi_1),\dots, (\bv_J,\xi_J) \)^k\) \) , \\
(\bar\bx,\bar t)^{k+1}  &= (\bx,t)^{k+1} + \theta \( (\bx,t)^{k+1} - (\bx,t)^k \) , \\
\( (\bv_1,\xi_1),\dots, (\bv_J,\xi_J)\)^{k+1} &= 
\prox_{\gamma_F \hat F}  
\(  \((\bv_1,\xi_1),\dots, (\bv_J,\xi_J)\)^k 
+ \gamma_F L (\bar\bx,\bar t)^{k+1} \), 
\end{aligned}
\end{equation}
where $\hat F$ is the Fenchel-Legendre transform of $F$, $\gamma_G > 0$ and $\gamma_F > 0$ are such that $\gamma_G\gamma_F < 1/\norm{L}^2$, and $\theta \in [-1,+\infty[$ (it is generally set to $\theta=1$ as in \cite{CP2011}).

\paragraph{Final remark about the primal-dual algorithm:} Note that the proposed approach computes directly the optimal affine mapping rather than computing the optimal subspace $V_n^{(\aff)}=\bar u + V_n$. This subspace is thus determined implicitly in view of Lemma \ref{proplin}, and we do not have any information about its dimension except that $1\leq \dim V_n \leq m$.

\section{Sensor placement}
\label{sec:sensor-placement}

In section \ref{sec:optimal-affine} we have summarized a strategy to find an optimal affine reconstruction algorithm $A^*_\aff$ for a given observation space $W$. This algorithm is connected to an optimal affine subspace $V_n^\opt$ to use in the PBDW method although we note that our procedure  does not yield an explicit characterization of $V_n^\opt$ and a further post-processing would be necessary to find it in practice. In \cite{BCMN2018}, we have considered the ``reciprocal'' problem, namely, for a given reduced model space $V_n$ with a good accuracy $\e_n$,   the question is how to guarantee a good reconstruction accuracy with a number of measurements $m\geq n$ as small possible. In view of the error bound \eqref{eq:pbdw-err-bound}, one natural objective is to guarantee that $\mu(V_n,W_m)$ is maintained of moderate size. Note that  taking $W_m=V_n$ would automatically give the minimal value $\mu(V_n,W_m)=1$ with $m=n$. However, in a typical data acquisition
scenario, the measurements that span the basis of $W_m$ are chosen from within a limited class. This is the 
case for example when placing $m$ pointwise sensors at various locations within the physical domain $\Omega$.

We model this restriction by asking that the $\ell_i$ are picked within
a {\it dictionary} $\cD$ of $V'$, that is a set of linear functionals
normalized according to 
$$
\|\ell\|_{V'}=1, \quad \ell\in \cD,
$$
which is {\em complete} in the sense that $\ell(v)=0$ for all $\ell\in \cD$ implies that $v=0$.
With an abuse of notation, we identify $\cD$ with
the subset of $V$ that consists of all
Riesz representers $\omega$ of the above linear functionals $\ell$.
With such an identification, $\cD$ is a set of functions
normalized according to 
$$
\|\omega\|=1, \quad \omega\in \cD,
$$
such that the finite linear combinations of elements of $\cD$ are
dense in $V$. Our task is therefore to pick $\{\omega_1,\dots,\omega_m\}\in \cD$ in such a way that 
\begin{equation}
\beta(V_n,W_m)\geq \beta^*>0,
\label{betastar}
\end{equation}
for some prescribed $0<\beta^*< 1$, with $m$ larger than $n$ but as small as possible. In particular, we may introduce
\begin{equation}
m^*=m^*(\beta^*,\cD,V_n),
\end{equation}
the minimal value of $m$ such that there exists  $\{\omega_1,\dots,\omega_m\}\in\cD$
satisfying \eqref{betastar}.

In \cite{BCMN2018} the authors show two ``extreme'' results:
\begin{itemize}
\item
For any $V_n$ and $\cD$, there exists $\beta^*>0$ such that $m^*=n$, that is,
the inf-sup condition \eqref{betastar} holds with the minimal possible number of measurements.
However this $\beta^*$ could be arbitrarily close to $0$.
\item
For any prescribed $\beta^*>0$ and any model space $V_n$,
there are instances of dictionaries $\cD$ such that $m^*$
is arbitrarily large.
\end{itemize}
The two above statements illustrate that the range of situations that can arise is very broad in full generality if one does not add extra assumptions on the nature of $V_n$ or on the nature of the dictionary $\cD$. This  motivates to analyse more concrete instances as we present next.

It is possible to study certain relevant dictionaries 
for the particular space $V=H^1_0(\Omega)$, with inner product and norms
 \begin{equation}
 \< u, v \> \coloneqq \int_\Omega  \nabla u(x) \cdot \nabla v(x) \, \mathrm{d}x
 \quad {\rm and} \quad \|u\|\coloneqq\|\nabla u\|_{L^2(\Omega)}.
 \label{vnorm}
 \end{equation}
The considered dictionaries
model local sensors, either as point evaluations or as local averages.
In the first case,
$$
\cD = \{ \ell_x = \delta_x \;:\; \forall x\in \Omega\},
$$
which requires that $V$ is a reproducing kernel Hilbert space (RKHS)
of functions defined on $\Omega$, that is a Hilbert space that continuously embeds in $\cC(\Omega)$.
Examples of such spaces are the Sobolev spaces $H^s(\Omega)$ for $s>d/2$, possibly
with additional boundary conditions. In the case of local averages, the linear functionals are of the form
\begin{equation}
\ell_{x,\tau}(u)=\int_\Omega u(y) \vp_\tau(y-x) dy,
\label{locav}
 \end{equation}
where 
\begin{equation}
\vp_\tau(y)\coloneqq\tau^{-d} \vp\(\frac y \tau\),
 \end{equation}
for some fixed radial function $\vp$ compactly supported in
the unit ball $B=\{|x|\leq 1\}$ of $\bR^d$ and such that  $\int \vp=1$, and $\tau>0$ representing
the point spread. The dictionary in this case is
$$
\cD = \{ \ell_{x, \tau} \;:\; \forall x\in \Omega\}.
$$
We could even consider an interval of values for $\tau$ in $[\tau_{\min}, \tau_{\max}]$ with $0< \tau_{\min}\leq \tau_{\max}$,
$$
\cD = \{ \ell_{x, \tau} \;:\; \forall (x, \tau) \in \Omega\times [\tau_{\min}, \tau_{\max}]\}.
$$
For the above cases of dictionaries, we
provide upper estimates of $m^*$ in the case of spaces $V_n$ that satisfy
some inverse estimates, such as finite element or trigonometric polynomial spaces. In \cite{BCMN2018}, the optimal value $m^*$ is proved to be of the same order as $n$
when the sensors are uniformly spaced.

This a-priori analysis is not possible for more general spaces $V$. It is not possible either for subspaces $V_n$ such as
reduced basis spaces, which are preferred to finite element spaces for model order reduction
because the approximation error $\e_n$ of the manifold $\cM$ defined in \eqref{eq:epsn} is expected to decay much faster in elliptic and parabolic problems. For such general spaces,
we need a strategy to select the measurements. In practice, $V$ is 
of finite but very large dimension and $\cD$ is of finite but very large cardinality
\begin{equation}
M\coloneqq\#(\cD)>\!\!>1.
\end{equation}
For this reason, the exhaustive search of the set $\{\omega_1,\dots,\omega_m\}\subset \cD$ maximizing
$\beta(V_n,W_m)$ for a given $m>1$ is out of reach. One natural alternative is to rely on greedy algorithms
where the $\omega_j$ are picked incrementally.

The starting point to the design of such algorithms is the observation that
\eqref{betastar} is equivalent to having
\begin{equation}
\sigma_m=\sigma(V_n,W_m)\coloneqq\sup_{v\in V_n, \|v\|=1} \|v-P_{W_m}v\|\leq \sigma^*,\quad \sigma^*\coloneqq\sqrt{1-(\beta^*)^2}<1.
\label{sigmam}
\end{equation}
Therefore, our objective is to construct a space $W_m$ spanned by $m$ elements from $\cD$ 
that captures all unit norm vectors of $V_n$ with the prescribed accuracy $\sigma^*<1$. This leads
us to study and analyze algorithms which may be thought as generalization to the well-studied
orthogonal matching pursuit algorithm (OMP), equivalent to the algorithms we study here when applied to 
the case $n=1$ with a unit norm vector $\phi_1$ that generates $V_1$. We refer to \cite{DT1996, TG2007, BCDD2008, Temlyakov2011} for some references on classical results on greedy algorithms and the OMP strategy.

In \cite{BCMN2018}, the authors propose and analyzed two algorithms which are summarized in Sections \ref{sec:collective-omp} and \ref{sec:worst-case-omp}. In Section \ref{sec:point-eval} the case of pointwise evaluations is discussed. The main result which is shown is that both algorithms always converge, ensuring that \eqref{betastar}
holds for $m$ sufficiently large, and we also give conditions on $\cD$ that allow us
to a-priori estimate the minimal value of $m$ where this happens.
The main observation stemming from numerical experiments is the ability of the greedy algorithms to pick good points. In particular, in the case of dictionaries of point evaluations or local averages, we observe that  the selection performed by the greedy algorithms is near optimal in simple 1D cases in the sense that it achieves \eqref{betastar} after a number of iterations which is proportional to $n$ and which can be predicted in theory.

Before finishing this section, let us outline the main differences and points of contact between the present approach and existing
works in the literature. The problem of optimal placement of sensors, which corresponds
to the particular setting where the linear functionals are point evaluations or local averages,
has been extensively studied since the 1970's in control and systems theory.
In this context, the state function to be estimated is the realization of a Gaussian stochastic process,
typically obtained as the solution of a linear PDE with a white noise forcing term. The error is then measured
in the mean square sense \eqref{eq:err-ms}, rather than in the worst case performance sense \eqref{eq:err-wc}
which is the point of view adopted in our work. The function to be minimized
by the sensors locations is then the trace of the error covariance, while we target at minimizing the inverse inf-sup
constant $\mu(V_n,W)$. See in particular \cite{Bensoussan1972} where the existence and characterization
of the optimal sensor location is established in this stochastic setting. Continuous optimization algorithms have been
proposed for computing the optimal sensor location, see e.g. \cite{AGI1975,CK1971,YS1973}.
One common feature with the present approach is that the criterion to be minimized by the optimal location is non-convex,
which leads to potential difficulties when the number of sensors is large. This is the main motivation for introducing a
greedy selection algorithm, which in addition allows us to consider more general dictionaries.

\subsection{A collective OMP algorithm}
\label{sec:collective-omp}

In this section we discuss a first numerical algorithm for the incremental selection of the
spaces $W_m$, inspired by the orthonormal matching pursuit (OMP) algorithm
which is recalled below. More precisely, our algorithm may be viewed as applying 
the OMP algorithm for the collective approximation of the elements of
an orthonormal basis of $V_n$ by linear combinations of $m$ members of the dictionary.

Our objective is to reach a bound \eqref{sigmam} for the quantity $\sigma_m$. Note that this
quantity can also be written as
$$
\sigma_m=\|(I-P_{W_m})|_{V_n} \|_{\cL(V_n,V)},
$$
that is, $\sigma_m$ is the spectral norm of $I-P_{W_m}$ restricted to $V_n$.

\paragraph{Description of the algorithm:}
When $n=1$, there is only one unit vector $\phi_1\in V_1$ up to a sign change. A commonly
used strategy for approximating $\phi_1$ by a small combination of elements from $\cD$ is 
to apply a greedy algorithm, the most prominent one being the orthogonal matching pursuit (OMP):
we iteratively select 
\begin{equation} \label{c-omp_iter}
\omega_k=\argmax_{\omega\in \cD}|\<\omega,\phi_1-P_{W_{k-1}} \phi_1\>|,
\end{equation}
where $W_{k-1}\coloneqq{\rm span}\{\omega_1,\dots,\omega_{k-1}\}$ and $W_0\coloneqq\{0\}$. In practice,
one often relaxes the above maximization, by taking $\omega_k$ such that
\begin{equation} \label{c-omp_iter_kappa}
|\<\omega_k,\phi_1-P_{W_{k-1}} \phi_1\>| \geq \kappa \max_{\omega\in \cD}|\<\omega,\phi_1-P_{W_{k-1}} \phi_1\>|,
\end{equation}
for some fixed $0<\kappa<1$, for example $\kappa=\frac 1 2$. This is known as the weak OMP algorithm,
but we refer to it as OMP as well.
It has been studied in \cite{BCDD2008,DT1996}, see also \cite{Temlyakov2011} for a complete survey on greedy approximation.

For a general value of $n$, one natural strategy is to define our greedy algorithm as follows: we iteratively select 
\begin{equation}
\omega_k=\argmax_{\omega\in \cD}\max_{v\in V_n, \|v\|=1}|\<\omega,v-P_{W_{k-1}} v\>|
=\argmax_{\omega\in \cD}\|P_{V_n}(\omega-P_{W_{k-1}} \omega)\|.
\label{nomp}
\end{equation}
Note that in the case $n=1$, we obtain the original OMP algorithm applied to $\phi_1$.

As to the implementation of this algorithm, we take $(\phi_1,\dots,\phi_n)$ to be
any orthonormal basis of $V_n$. Then
$$
\|P_{V_n}(\omega-P_{W_{k-1}} \omega)\|^2=\sum_{i=1}^n |\<\omega-P_{W_{k-1}} \omega, \phi_i\>|^2
=\sum_{i=1}^n |\<\phi_i-P_{W_{k-1}} \phi_i, \omega\>|^2
$$
Therefore, at every step $k$, we have
$$
\omega_k=\argmax_{\omega\in \cD}\sum_{i=1}^n |\<\phi_i-P_{W_{k-1}} \phi_i, \omega\>|^2,
$$
which amounts to a stepwise optimization of a similar nature as in
the standard OMP. Note that, while the basis $(\phi_1,\dots,\phi_n)$ is used for the implementation, 
the actual definition of the greedy selection algorithm is independent of the 
choice of this basis in view of \eqref{nomp}. It only involves 
$V_n$ and the dictionary $\cD$. Similar to OMP, we may weaken the algorithm by taking
$\omega_k$ such that
$$
\sum_{i=1}^n |\<\phi_i-P_{W_{k-1}} \phi_i, \omega_k\>|^2
\geq \kappa^2 \max_{\omega\in \cD}\sum_{i=1}^n |\<\phi_i-P_{W_{k-1}} \phi_i, \omega\>|^2,
$$
for some fixed $0<\kappa<1$.

For such a basis, we introduce the residual quantity
$$
r_m\coloneqq\sum_{i=1}^n  \|\phi_i-P_{W_m}\phi_i\|^2.
$$
This quantity allows us to control the validity of \eqref{betastar} since we have
$$
\sigma_m=\sup_{v\in V_n, \|v\|=1} \|v-P_{W_m}v\|=
\sup_{\sum_{i=1}^n c_i^2=1}\Big\| \sum_{i=1}^n c_i (\phi_i-P_{W_m}\phi_i)\Big\|
\leq r_m^{1/2},
$$
and therefore \eqref{betastar} holds provided that $r_m\leq \sigma^2=1-\gamma^2$.


\paragraph{Convergence analysis:}
By analogy to the analysis of OMP provided in \cite{DT1996}, we
introduce for any $\Psi=(\psi_1,\dots,\psi_n)\in V^n$ the quantity
$$
\|\Psi\|_{\ell^1(\cD)}\coloneqq\inf_{c_{\omega,i}} \Big \{\sum_{\omega\in\cD} \(\sum_{i=1}^n |c_{\omega,i}|^2\)^{1/2}\; : \; 
\psi_i=\sum_{\omega\in\cD} c_{\omega,i}\omega, \quad i=1,\dots, n\Big\},
$$
or equivalently, denoting $c_\omega\coloneqq\{c_{\omega,i}\}_{i=1}^n$,
$$
\|\Psi\|_{\ell^1(\cD)}\coloneqq\inf_{c_\omega} \Big \{\sum_{\omega\in\cD} \|c_{\omega}\|_2 \; : \; 
\Psi=\sum_{\omega\in\cD} c_\omega \omega \Big\}.
$$
This quantity is a norm on the subspace of $V^n$ on which it is finite.

Given that $\Phi=(\phi_1,\dots,\phi_n)$ is any orthonormal basis of $V_n$, we
write
$$
J(V_n)\coloneqq\|\Phi\|_{\ell^1(\cD)}.
$$
This quantity is indeed independent on the orthonormal basis $\Phi$:
if $\tilde \Phi=(\tilde \phi_1,\dots,\tilde \phi_n)$ is another orthonormal basis, we have $\tilde \Phi=U\Phi$ where
$U$ is unitary. Therefore any representation $\Phi=\sum_{\omega\in\cD} c_\omega \omega$
induces the representation 
$$
\tilde \Phi=\sum_{\omega\in\cD} \tilde c_\omega \omega, \quad \tilde c_\omega=Uc_\omega,
$$
with the equality
$$
\sum_{\omega\in\cD} \|\tilde c_{\omega}\|_2=\sum_{\omega\in\cD} \|c_{\omega}\|_2,
$$
so that $\|\Phi\|_{\ell^1(\cD)}=\|\tilde \Phi\|_{\ell^1(\cD)}$.

One important observation is that if 
$\Phi=(\phi_1,\dots,\phi_n)$ is an orthonormal basis of $V_n$
and if $\Phi=\sum_{\omega\in\cD} c_\omega \omega$, one has
$$
n=\sum_{i=1}^n \|\phi_i\| \leq \sum_{i=1}^n \sum_{\omega\in\cD} |c_{\omega,i}|
=\sum_{\omega\in\cD} \|c_{\omega}\|_1
\leq \sum_{\omega\in\cD} n^{1/2} \|c_\omega\|_2.
$$
Therefore, we always have
$$
J(V_n)\geq n^{1/2}.
$$
Using the quantity $J(V_n)$,
we can generalize the result of \cite{DT1996} on the OMP algorithm in the following way.

\begin{theorem}
\label{theo1}
Assuming that $J(V_n)<\infty$, the collective OMP algorithm satisfies
\begin{equation}
r_m
\leq \frac {J(V_n)^2}{\kappa^2} (m+1)^{-1}, \quad m\geq 0.
\label{rate-1}
\end{equation}
\end{theorem}

\begin{remark}
Note that the right side of \eqref{rate-1}, is always larger than $n(m+1)^{-1}$, which 
is consistent with the fact that $\beta(V_n,W_m)=0$ if $m<n$.
\end{remark}

One natural strategy for selecting the measurement space $W_m$
is therefore to apply the above described greedy algorithm, until the
first value $\tilde m=\tilde m(n)$ is met such that $\beta(V_n,W_m)\geq \gamma$.
According to \eqref{rate-1}, this value satisfies
\begin{equation}
\label{m-collectiveOMP}
m(n)\leq \frac {J(V_n)^2}{\kappa^2 \sigma^2}.
\end{equation}
For a general dictionary $\cD$ and space $V_n$
we have no control on the quantity $J(V_n)$ which could even be infinite, and therefore 
the above result does not guarantee that the above selection strategy 
eventually meets the target bound $\beta(V_n,W_m)\geq \gamma$.
In order to treat this case, we establish a perturbation result
similar to that obtained in \cite{BCDD2008} for the standard OMP algorithm.

\begin{theorem}
\label{theo2}
Let $\Phi=(\phi_1,\dots,\phi_n)$ be an orthonormal basis of $V_n$
and $\Psi=(\psi_1,\dots,\psi_n)\in V^n$ be arbitrary. Then the application
of the collective OMP algorithm on the space $V_n$ gives 
\begin{equation}
r_m \leq 4 \frac {\|\Psi\|_{\ell^1(\cD)}^2}{\kappa^2}(m+1)^{-1}+\|\Phi-\Psi\|^2, \quad m\geq 1.
\label{ratepert-th2}
\end{equation}
where $\|\Phi-\Psi\|^2\coloneqq\|\Phi-\Psi\|^2_{V^n}=\sum_{i=1}^n \|\phi_i-\psi_i\|^2$.
\end{theorem}

As an immediate consequence of the above result, we obtain 
that the collective OMP converges for any space $V_n$, even when $J(V_n)$
is not finite.

The next corollary shows that if $\gamma>0$, one has
$\beta(V_n,W_m)\geq \gamma$ for $m$ large enough.

\begin{corollary}
\label{corconv}
For any $n$ dimensional space $V_n$, the application
of the collective OMP algorithm on the space $V_n$ gives that $\lim_{m\to +\infty} r_m=0$. 
\end{corollary}

\subsection{A worst case OMP algorithm}
\label{sec:worst-case-omp}
We present in this section a variant of the previous collective OMP algorithm  first tested in \cite{MPPY2015}, and then analyzed in \cite{BCMN2018}.
In numerical experiments this variant performs 
better than the collective OMP algorithm,
however its analysis is more delicate. In particular we do not 
obtain convergence bounds that are as good.

\paragraph{Description of the algorithm:}

We first take 
\begin{equation} \label{wc-omp_wk_iter}
v_k\coloneqq{\rm argmax}\Big\{\|v-P_{W_{k-1}} v\| \, : \, v\in V_n, \, \|v\|=1\Big\},
\end{equation}
the vector in the unit ball of $V_n$ that is less well captured by $W_{k-1}$
and then define $\omega_{k}$ by applying one step of OMP to this vector,
that is 
\begin{equation} \label{wc-omp_iter_kappa}
| \<v_k-P_{W_{k-1}} v_k,\omega_{k}\>| \geq \kappa {\max}\Big\{| \<v_k-P_{W_{k-1}} v_k,\omega\>| \, :\, \omega\in \cD\Big\},
\end{equation}
for some fixed $0<\kappa<1$.


\paragraph{Convergence analysis:}

The first result
gives a convergence rate of $r_m$ under the assumption
that  $J(V_n)<\infty$, similar to Theorem \ref{theo1}, however
with a multiplicative constant that is inflated by $n^2$.

\begin{theorem}
\label{theo3}
Assuming that $J(V_n)<\infty$, the worst case OMP algorithm satisfies
\begin{equation}
r_m
\leq \frac {n^2J(V_n)^2}{\kappa^2} (m+1)^{-1}, \quad m\geq 0.
\label{rateworse1}
\end{equation}
\end{theorem}

For the general case, we establish a perturbation result
similar to Theorem \ref{theo2}, with again a multiplicative constant 
that depends on the dimension of $V_n$.

\begin{theorem}
\label{theo4}
Let $\Phi=(\phi_1,\dots,\phi_n)$ be an orthonormal basis of $V_n$
and $\Psi=(\psi_1,\dots,\psi_n)\in V^n$ be arbitrary. Then the application
of the worst case OMP algorithm on the space $V_n$ gives 
\begin{equation}
r_m \leq 4 \frac {n^2\|\Psi\|_{\ell^1(\cD)}^2}{\kappa^2}(m+1)^{-1}+n^2\|\Phi-\Psi\|^2, \quad m\geq 1.
\label{ratepert-th4}
\end{equation}
where $\|\Phi-\Psi\|^2\coloneqq\|\Phi-\Psi\|^2_{V^n}=\sum_{i=1}^n \|\phi_i-\psi_i\|^2$.
\end{theorem}

By the exact same arguments as in the previous algorithm, we find that
that the worst case OMP converges for any space $V_n$, even when $J(V_n)$
is not finite.

\begin{corollary}
\label{corconvworse}
For any $n$ dimensional space $V_n$, the application
of the worst case OMP algorithm on the space $V_n$ gives that $\lim_{m\to +\infty} r_m=0$. 
\end{corollary}

\subsection{Application to point evaluation}
\label{sec:point-eval}

As a simple example, we consider a bounded univariate interval $\Omega = I$  and take
$V=H^1_0(I)$ which is continuously embedded in 
$\cC(I)$. Without loss of generality we take $I=]0,1[$.  For every $x\in ]0,1[$, 
the Riesz representer of $\delta_{x}$ is given by the solution of $\omega''=\delta_x$ with 
zero boundary condition. Normalising this solution $\omega$ it with respect to the $V$ norm, we obtain
\begin{equation}
\label{eval_rep_H1}
\omega_{x}(t) =
\begin{cases}
\frac{t(1-x)}{\sqrt{x(1-x)}},\quad \text{for } t\leq x \\
\frac{(1-t)x}{\sqrt{x(1-x)}},\quad \text{for } t> x.
\end{cases}
\end{equation}

For any set of $m$ distinct points $0< x_1 < \dots <x_m< 1$, the associated measurement space $W_m=\vspan\{ \omega_{x_1},\dots, \omega_{x_m}\}$ coincides with the space of piecewise affine polynomials with nodes at $x_1,\dots,x_m$ that vanish at the boundary. Denoting $x_0\coloneqq0$ and $x_{m+1}\coloneqq1$, we have
\begin{equation} \label{Wm_equiv}
W_m = \{ \omega \in \cC^0([0,1]),\ \omega|_{[x_k,x_{k+1}]} \in \mathbb P_1,\ 0\leq k \leq m,\text{ and }\omega(0)=\omega(1)=0 \}.
\end{equation}
As an example for the space $V_n$, let us consider the
span of the Fourier basis (here orthonormalized in $V$),
\begin{equation}
\phi_k \coloneqq \frac{\sqrt{2}}{\pi k} \sin(k\pi x),\quad 1\leq k\leq n.
\label{phik}
\end{equation}

Let us now estimate $m(n)$ in this example if we choose the points with the greedy algorithms that we have introduced. This boils down to estimate for $J(V_n)$. In this simple case,
\begin{equation*}
J(V_n)\coloneqq\|\Phi\|_{\ell^1(\cD)}=\inf \Big \{\int_{x\in[0,1]} \|c_x\|_2\,\dx \; : \; 
\Phi=\int_{x\in[0,1]} c_x \omega_x\,\dx \Big\}
\end{equation*}
and we can derive $c_x$ for every $x\in [0,1]$ by differentiating twice the components of $\Phi$ since
\begin{equation*}
\Phi''(x)=\int_{y\in[0,1]} c_y \omega_y''(x)\,\dy= -\int_{y\in[0,1]} c_y \delta_y(x)\,\dx=-c_x.
\end{equation*}
Thus, using the basis functions $\phi_k$ defined by \eqref{phik}, we have
\begin{equation*}
J(V_n)
=\int_{x\in[0,1]} \left( \sum_{k=1}^n |\phi_k''(x)^2|\right)^{1/2}\,\dx
= \int_{x\in[0,1]} \left( \sum_{k=1}^n 2k\pi |\sin(k\pi x)|^2\right)^{1/2}\,\dx
\sim n^{3/2}.
\end{equation*}
Estimate \eqref{m-collectiveOMP} for the convergence of the collective OMP approach yields
\begin{equation*}
m(n)\gtrsim \frac{n^3}{\kappa^2\sigma^2},
\end{equation*}
while for the worst case OMP, estimate \eqref{rateworse1} gives
\begin{equation*}
m(n)\gtrsim \frac{n^5}{\kappa^2\sigma^2}.
\end{equation*}
These bounds deviate from the optimal estimation due to the use of the Hilbert-Schmidt norm in the analysis. Numerical results reported in \cite{BCMN2018}  reveal that the greedy algorithms actually behave much better in this case.

\section{Joint selection of $V_n$ and $W_m$}
\label{sec:GEIM}

\subsection{Optimality benchmark}
So far, we have studied linear and affine reconstruction algorithms which involve an affine reduced model space $V^{(\aff)}_n$ and an observation space $W_m$. In Section \ref{sec:optimal-affine} we have fixed the observation space, and we have discussed how to derive the optimal $V^{(\aff)}_n$, which is directly connected to the optimal affine algorithm of the benchmark that we have introduced in \eqref{eq:opt-aff-alg}. In Section \ref{sec:sensor-placement} we have examined the ``reciprocal'' of this problem, namely the case where we fix $V_n$ and we select sensor measurements $\omega_i$ from a dictionary $\cD$. The selection is done in order to build an observation space $W_m=\vspan\{\omega_i\}_{i=1}^m$ that yields stable reconstructions in the sense of minimizing $\mu(V_n, W_m)$ (or, equivalently, maximizing $\beta(V_n, W_m)$).

One can of course envision a combined approach in which we make a joint selection of $V_n$ and $W_m$. Of course, the basis $\{\omega_i\}_{i=1}^m$ spanning $W_m$ must be selected from a dictionary $\cD$ in order to account for the fact that we are working with sensor measurements. One way of defining the best performance that such a joint selection can deliver is given by the following extension of the benchmark \eqref{eq:opt-aff-alg}. For a fixed $m\geq 1$, the optimal performance of the joint approach is
\begin{equation}
\label{eq:opt-joint-aff-alg}
E^*_{\wca, \joint}(\cM, m)\;  = \;\min_{ \{\omega_i\}_{i=1}^m \in \cD^n } 
\underset{ {\substack{A: \vspan\{ \omega_i \}_{i=1}^n \to V \\ A \text{ affine}}}}{\min}\; E_{\wc}(A,\cM),
\end{equation}
for the case of affine algorithms. Of course, one can similarly define the best performance among all algorithms (affine and nonlinear) by removing the constraint that $A$ is affine in the definition above, that is,
\begin{equation}
\label{eq:opt-joint-alg}
E^*_{\wc, \joint}(\cM, m)\;  = \;\min_{ \{\omega_i\}_{i=1}^m \in \cD^n } 
\underset{ {\substack{A: \vspan\{ \omega_i \}_{i=1}^n \to V \\ A \text{ affine}}}}{\min}\; E_{\wc}(A,\cM).
\end{equation}

\subsection{A general nested greedy algorithm}
\label{sec:nested-greedy}
Finding the optimal elements $\{\omega_i^*\}_{i=1}^m$ and the optimal algorithm
$$
A^*:  W^* \to V,\quad W^* \coloneqq \vspan\{\omega_i^*\}_{i=1}^m
$$
that meet \eqref{eq:opt-joint-aff-alg} or \eqref{eq:opt-joint-alg} is a very difficult task, and, to best of the author's knowledge, this question remains an open problem. There are however a number of practical algorithms that have been proposed in order to perform a satisfactory joint selection of $V_n$ and $W_m$ in the framework of affine reconstruction algorithms (see, e.g., \cite{MMPY2015, MPPY2015, BCMN2018}). All strategies are based on nested greedy algorithms, and they can be seen as variations of the following general algorithm.

Assume that we have fixed a dictionary $\cD$ to select the sensors. Fix a minimal admissible value for the inf-sup stability $\underline{\beta}>0$. For $n=1$, select
$$
u_1 \in \argmax_{u\in \cM} \Vert u \Vert
$$
and set
$$
V_1 \coloneqq \vspan \{ u_1 \}.
$$
For the given $V_1$, apply the OMP sensor selection strategy from Section \ref{sec:collective-omp} or its variant from Section \ref{sec:worst-case-omp}. At every iteration $k\geq 1$ of the OMP, we pick an observation function $\omega_k^1$. The iterations stop as soon as we reach a value $k=m(1)$ such that
$$
\beta(V_1, \vspan\{\omega_k\}_{k=1}^{m(1)}) \geq \underline{\beta}.
$$
We then set
$$
\cO_1 \coloneqq \{\omega_k\}_{k=1}^{m(1)},
\qquad \text{and} \qquad
W_{m(1)} \coloneqq \vspan \{ \cO_1 \}.
$$

We next proceed by induction. At step $n>1$, assume that we have selected:
\begin{itemize}
\item the set of functions $\{u_1,\dots, u_{n-1}\}$ spanning $V_{n-1} \coloneqq \vspan \{u_1,\dots, u_{n-1}\}$,
\item the set of observation functions $\cup_{i=1}^{n-1} \cO_i $ spanning
$$
W_{m(n-1)} \coloneqq \vspan\{ \cup_{i=1}^{n-1} \cO_i \}.
$$
\end{itemize} 
We select the next function $u_n$ and the set $\cO_n$ of observation functions as follows. Consider the linear PBDW reconstruction algorithm $A_{n-1}$ associated to the spaces $V_{n-1}$ and $W_{m(n-1)}$ and find
$$
u_n \in \argmax_{u\in \cM} \Vert u - A_{n-1}(P_{W_{m(n-1)}}u)\Vert.
$$
We next define
$$
V_n \coloneqq \vspan \{ u_i \}_{i=1}^n = V_{n-1} + \vspan\{u_n\}.
$$

If $\beta(V_n, W_{ m(n-1) }) \geq \underline{\beta}$, the stability condition is satisfied at step $n$ without needing to add any extra observation functions. As a consequence, we set $\cO_n = \emptyset$. Then we define
$$
W_{m(n)} = \vspan\{ \cup_{i=1}^n \cO_i \} = W_{m(n-1)},
$$
and go to step $n+1$.

If $\beta(V_n, W_{ m(n-1) }) < \underline{\beta}$, we apply the OMP strategy for the constructed $V_n$, taking $W_{m(n-1)}$ as the initial measurement space to which we have to add new dimensions. For example, in the case of the worst case OMP, we iteratively select for $k\geq 1$
\begin{equation}
\omega^n_k =\argmax_{\omega\in \cD}\|P_{V_n}(\omega-P_{W_{m(n-1)}+\vspan\{ \omega_i \}_{i=1}^{k-1}  } \omega)\|.
\end{equation}
and we stop the iterations as soon as we reach a value $k=m(n)$ such that
$$
\beta \left(V_n, W_{m(n-1)}+\vspan\{ \omega_i \}_{i=1}^{k}  \right) \geq \underline{\beta}.
$$
Once this criterion is satisfied, we set 
$$
\cO_{n} \coloneqq \{ \omega^n_k  \}_{k=1}^{m(n)}
$$
and we finish iteration $n$ by defining
$$
W_{m(n)} = \vspan\{ \cup_{i=1}^n \cO_i \}.
$$
As a termination criterion for our algorithm, we can stop the outer iterations in $n$ as soon as
$$
\max_{u\in \cM} \Vert u - A_{n-1}(P_{W_{m(n-1)}}u)\Vert < \eps
$$
for a given prescribed tolerance $\eps>0$.

A straightforward application of the results proven in \cite{BCDDPW2011, MMT2016} leads to the following result. It expresses the fact that the reconstruction error with the spaces $V_n$ and $W_{m(n)}$ decays at a comparable rate as the Kolmogorov $n$-width.

\begin{theorem}
Let $A_n$ be the linear PBDW algorithm associated to the spaces  $V_n$ and $W_{m(n)}$ built with the nested greedy algorithm. Then, for $a, b, q \in \bR^*_+$,
\begin{equation}
\begin{cases}
d_n(\cM) &\lesssim n^{-q}  \\
d_n(\cM) &\lesssim e^{-a n^{b}}
\end{cases}
\quad \Rightarrow \quad
\begin{cases}
E_\wc(\cM, A_n) &\lesssim n^{-q} \\
E_\wc(\cM, A_n) &\lesssim e^{-\tilde a n^{\tilde b}},
\end{cases}
\end{equation}
where $\tilde b = \frac{b}{b+1}$, and $\tilde a$ depends on $a$ and some other technical parameters.

\end{theorem}

\subsection{The Generalized Empirical Interpolation Method}
\label{sec:geim-algo}
Among the many variants that one can consider of the above joint selection strategy, one that has drawn particular attention is the so-called Generalized Empirical Interpolation Method (GEIM, \cite{MM2013, MMPY2015, MMT2016}). In this method, at every step $n\geq 1$, we add only one observation function. The criterion to select it is close (but not entirely equivalent) to the one of making one single step of the worst case OMP of Section \ref{sec:worst-case-omp}. This implies that we prescribe $m(n) = n$ for all $n\geq 1$, and the dimension of the reduced model $V_n$ is equal to the one of the observation space $W_n$. One consequence of this construction is that one cannot guarantee that $\beta(V_n, W_n)$ remains bounded away from $0$. This is in contrast to the algorithm of Section \ref{sec:nested-greedy}. In practice, it has been observed that $\beta(V_n, W_n)$ slowly decreases as $n\to \infty$ (see, e.g., \cite{MM2013, MMPY2015}) but there is no a priori analysis quantifying the rate of decay. 

The algorithm works as follows (see, e.g., \cite{MMPY2015}). For $n=1$, select
$$
u_1 \in \argmax_{u\in \cM} \Vert u \Vert
$$
and set
$$
V_1 \coloneqq \vspan \{ u_1 \}.
$$
The first observation function is defined as
$$
\omega_1 \in \argmax_{\omega \in \cD} |\left< \omega, u_1 \right> |,
$$
and we set
$$
W_1 \coloneqq \vspan\{ \omega_1\}.
$$
We then proceed by induction. At step $n>1$, assume that we have selected $\{u_1,\dots, u_{n-1} \}$ and $\{\omega_1,\dots, \omega_{n-1}\}$ which respectively span the subspaces $V_{n-1}$ and $W_{n-1}$. We define $A_{n-1}$ as the PBDW reconstruction algorithm associated to $V_{n-1}$ and $W_{n-1}$. We choose
$$
u_n \in \argmax_{u\in \cM} \Vert u - A_n(P_{W_{n-1}} u) \Vert,
$$
and then select
$$
\omega_n \in \argmax_{\omega \in \cD}  | \left< \omega, u_n - A_{n-1}(P_{W_{n-1}} u) \right> |.
$$
We finally define
$$
V_n \coloneqq V_{n-1} + \vspan\{ u_n \},
\qquad \text{and} \qquad
W_n \coloneqq W_{n-1} + \vspan\{ \omega_n\},
$$
and go the next step $n+1$.

The method is called generalized interpolation because we have the interpolatory property that $\ell_i(v) = \ell_i(A_n(v))$ for $i=1,\dots, n$. Also, for any $v\in V_n$, $A_n(v)=v$.

\section{A Piece-Wise Affine Algorithm to reach the Benchmark Optimality}
\label{sec:piecewise-affine}

In this section, we come back to the setting where we work with a fixed observation space $W$ and a fixed number $m$ of observations $z_i = \ell_i(u), i=1, \dots, m$.  Our goal is to discuss how to go beyond the linear/affine framework that we have discussed in sections  \ref{sec:optimal-affine} to \ref{sec:GEIM}, and how to build algorithms that can deliver a performance close to optimal.

The simplicity of the plain PBDW method \eqref{eq:pbdw-lin} and its above variants come together with a 
fundamental limitation of performance: since the map $w\mapsto A_n(w)$ is linear or affine,
the reconstruction necessarily belongs to an $m$ or $m+1$ dimensional space, and therefore the worst case performance
is necessarily bounded from below by the Kolmogorov width $d_m(\cM)$ or $d_{m+1}(\cM)$. In other words, if we restrict ourselves to affine algorithms, we have
\begin{equation}
\underset{A:W\to V}{\min}\; E_{\wc}(A,\cM)
\leq
d_{m+1}(\cM)
\leq
\underset{ {\substack{A:W\to V \\ A \text{ affine}}}}{\min}\; E_{\wc}(A,\cM).
\end{equation}
and affine algorithms will miss optimality especially in cases where
\begin{equation}
\underset{A:W\to V}{\min}\; E_{\wc}(A,\cM)
\ll
d_{m+1}(\cM).
\end{equation}
This is expected to happen in elliptic problems with weak coercivity or in hyperbolic problems. 

In view of this limitation, the principal objective of \cite{CDMN2020}
is to develop {\it nonlinear} state estimation techniques which \emph{provably} overcome
the bottleneck of the Kolmogorov width $d_m(\cM)$. The next pages summarize the main ideas
from this contribution. We will focus particularly on summarizing a nonlinear recovery method based on a family of affine reduced models $(V_k)_{k=1,\dots,K}$.
Each $V_k$ has dimension $n_k\leq m$ and serves 
as a local approximation to a portion $\cM_k$ of the solution manifold. 
Applying the PBDW method with each such space, results in a collection 
of state estimators $u^*_k$. The value $k$ for which the true state $u$ belongs
to $\cM_k$ being unknown, we introduce a {\em model selection} procedure in order 
to pick a value $k^*$, and define the resulting estimator $u^*=u^*_{k^*}$.
We show that this estimator has performance comparable to  
optimal in a sense which we make precise later on, and which cannot be achieved 
by the standard linear/affine PBDW method due to the above described limitations.

Model selection is a classical topic of mathematical statistics \cite{Massart2007},
with representative techniques such as complexity penalization or cross-validation
in which the data are used to select a proper model. The approach that we present differs
from these techniques in that it exploits  {(in the spirit of {\em data assimilation})} the PDE model which is available
to us, by evaluating the distance to the manifold
\begin{equation}
\dist(v,\cM)=\min_{y\in Y} \|v-u(y)\|,
\label{distM}
\end{equation}
of the different estimators $v=u^*_k$ for $k=1,\dots,K$, and picking the value
$k^*$ that minimizes it. In practice, the quantity \eqref{distM}
cannot be exactly computed and we instead rely on a computable surrogate quantity 
${\cal S}(v,\cM)$ expressed in terms of the residual to the PDE. One typical instance 
where such a surrogate is available and easily computable is when the parametric PDE 
has the form of a linear operator equation
\begin{equation}
\cB(y)u=f(y),
\end{equation}
where $\cB(y)$ is boundedly invertible from $V$ to $V'$, or more generally, from $V\to Z'$ for a test space $Z$ different from $V$, uniformly over $y\in Y$. Then ${\cal S}(v,\cM)$ is obtained
by minimizing the residual 
\begin{equation}
\cR(v,y)=\|\cB(y)v{-}f(y)\|_{Z'},
\end{equation}
over $y\in Y$. In other words,
$$
{\cal S}(v,\cM) = \min_{y\in \rY} \cR(v,y).
$$
This task itself is greatly facilitated in the case where the operators $\cB(y)$ and
source terms $f(y)$ have affine dependence in $\rY$. One relevant example is the
second order elliptic diffusion equation with affine diffusion coefficient,
\begin{equation}
-{\rm div}(a\nabla u)=f(y), \quad a=a(x; y)=\bar a(x)+\sum_{j=1}^d y_j\psi_j(x).
\label{ellip}
\end{equation}

%

%
%
%
%

\subsection{Optimality benchmark under perturbations}
In order to present the piece-wise affine strategy and its performance, we need to enrich the notions of benchmark optimality introduced in section \ref{sec:optim-benchmarks}. In that section, we introduced in \eqref{eq:benchmark-M} the quantity $\delta_0$ which was defined as
\begin{equation}
\delta_0=\delta_0(\cM,W)\coloneqq {\sup}\{{\rm diam}(\cM_w)\,:\, w\in W\}=\sup \{\|u-v\|\; : \; u,v\in \cM, \;u-v\in W^\bot \}.
\end{equation}
We saw in \eqref{eq:equiv-optimality} that $\delta_0$ can be related to the worst-case optimal performance $E^*_\wc(\cM)$ by the equivalence
$$
\frac 1 2 \delta_0 \leq E^*_\wc(\cM) \leq \delta_0.
$$
We next introduce a somewhat relaxed benchmark quantity to take into account the fact that computationally feasible algorithms usually introduce simplifications of the geometry of the manifold. In the case of the plain PBDW, the simplification is that the manifold is ``replaced'' by a linear or an affine subspace $V_n$, which makes that for most practical and theoretical purposes, $\cM$ could be replaced by the cylinder $\cK_n$ introduced in \eqref{eq:cylinder}. As we will see later on, the relaxed benchmark will also allow us to take into account model error and measurement noise in the analysis.

In order to account for manifold simplification as well as model bias,
for any given accucary $\sigma>0$, we introduce the $\sigma$-offset of $\cM$,
\begin{equation}
\cM_\sigma\coloneqq \{v\in V\; : \; \dist(v,\cM) \leq \sigma\}=\bigcup_{u\in\cM} B(u,\sigma),
\end{equation}
where $B(u,\sigma)$ is the ball of center $u$ and radius $\sigma$.
Likewise, we introduce the set 
\begin{equation}
\cM_{\sigma,w} =\cM_\sigma\cap (\omega + W^\perp),
\label{msigmaw}
\end{equation}
which is a perturbed set of $\cM_w$ introduced in \eqref{eq:M-omega} (note that this set still excludes uncertainties in $w$ but we will come to this in a moment).

Our benchmark for the worst case error is now defined as
\begin{equation}
\delta_\sigma\coloneqq \max_{w\in W} {\rm diam}(\cM_{\sigma,w})=\max \{\|u-v\|\; : \; u,v\in \cM_\sigma, \;u-v\in W^\bot \}.
\label{benchapp}
\end{equation}
Figures \ref{fig:bird2} and \ref{fig:bird-offset} give an illustration of $\delta_0$, $\delta_\sigma$ and the optimal scheme $A^*_\wc$ based on Chebyshev centers which was introduced in \eqref{eq:optimal-A-Cheb-center}.
\begin{figure}[ht]
\begin{subfigure}{.5\textwidth}
  \centering
  \includegraphics[width=\textwidth]{figures/state-estimation/bird-shape-paper.png}
  \caption{Perfect Physical Model + No Noise.}
  \label{fig:bird2}
\end{subfigure}
\begin{subfigure}{.5\textwidth}
  \centering
  \includegraphics[width=\textwidth]{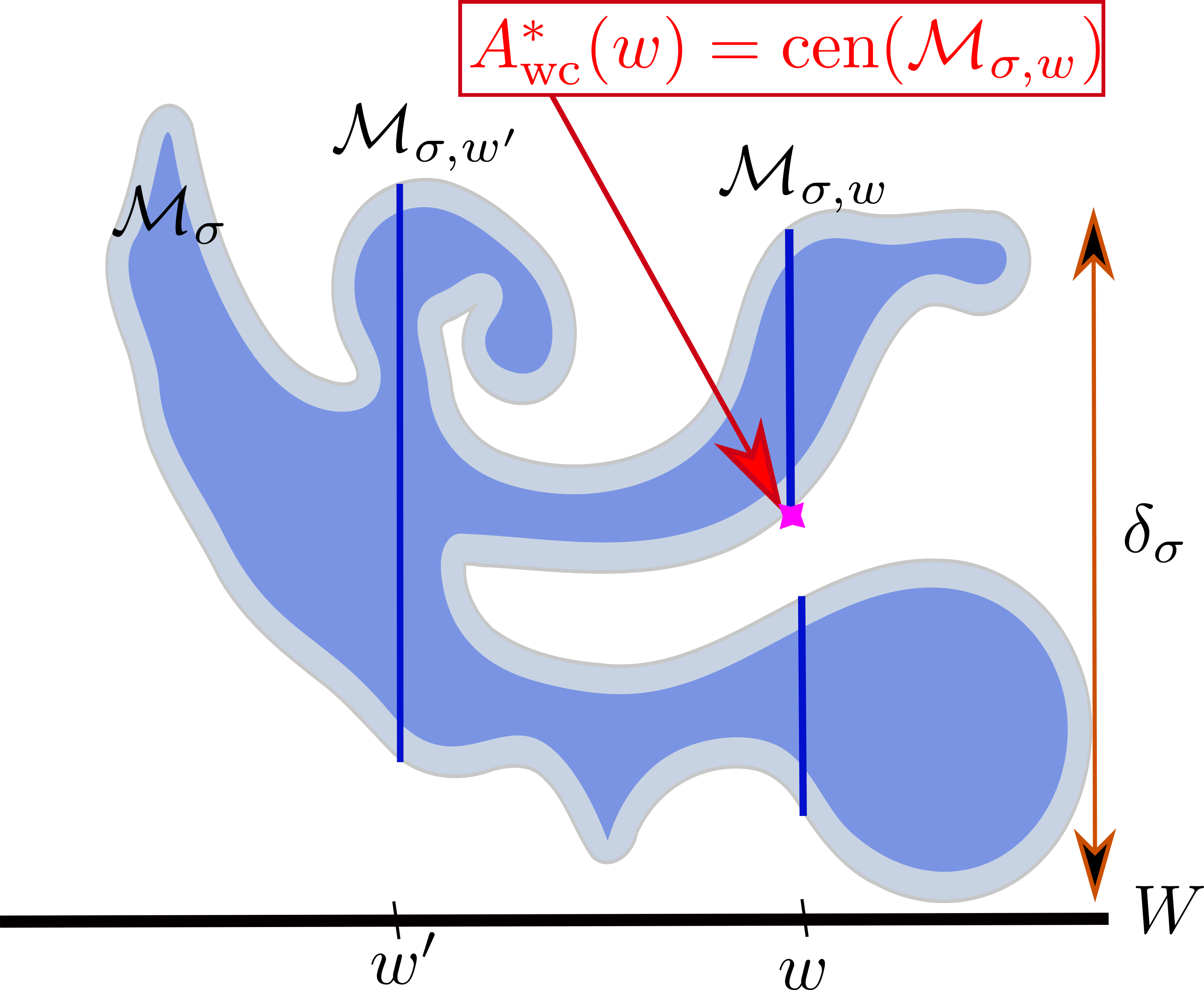}
  \caption{Inexact Physical Model + Noise.}
  \label{fig:bird-offset}
\end{subfigure}
\caption{Illustration of the optimal recovery benchmark on a manifold in the two dimensional Euclidean space. Left: benchmark for the idea scenario of a perfect model and noiseless observations. Right: how the benchmark is degraded by an abstract factor $\sigma$ associated to the model error and the observation noise.}
\label{fig:fig}
\end{figure}

To account for measurement noise, we introduce the quantity
\begin{equation}
\tilde \delta_\sigma\coloneqq \max \{\|u-v\|\; : \; u,v\in \cM,\; \|P_W u-P_W v\|\leq \sigma \}.
\end{equation}
{The two quantities $\delta_\sigma$ and $\tilde \delta_\sigma$ are not equivalent, however}
one has the framing
\begin{equation}
\delta_\sigma-2\sigma\leq \tilde \delta_{2\sigma}\leq \delta_\sigma+2\sigma.
\label{framing}
\end{equation}

In {the following}  analysis of reconstruction methods,
we use the quantity $\delta_\sigma$ as a benchmark which, in view of this last observation, also
accounts for the lack of accuracy in the measurement of $P_Wu$.
Our objective is therefore to design an algorithm that, for a given tolerance $\sigma>0$,
recovers from the measurement $w=P_Wu$ an approximation to $u$
with accuracy comparable to $\delta_\sigma$. Such an algorithm requires 
that we are able to capture the solution manifold up to some tolerance 
$\e\leq \sigma$ by some reduced model. 

\subsection{Piecewise affine reduced models}

Linear or affine reduced models, as used in the affine PBDW algorithm,
are not suitable for approximating the solution manifold
when the required tolerance $\e$ is too small.
In particular, when $\e{<} d_m(\cM)$ one would then need to use a linear space $V_n$ of 
dimension $n{>} m$, therefore {making} $\mu(V_n,W)$ infinite.

One way out is to replace the single space $V_n$ by a 
{\em family} of affine spaces 
\begin{equation}
V_k=\bar u_k+\bar V_k, \quad k=1,\dots,K,
\end{equation}
each of them having dimension 
\begin{equation}
\dim(V_k)=n_k\leq m,
\end{equation}
such that the manifold is well captured by the union of these spaces,
in the sense that 
\begin{equation}
{\rm dist}\(\cM,\bigcup_{k=1}^K V_k\)\leq \e
\end{equation}
for some prescribed tolerance $\e>0$. This is equivalent to saying that there exists a partition
of the solution manifold
\begin{equation}
\cM=\bigcup_{k=1}^K \cM_k,
\end{equation}
such that we have local certified bounds
\begin{equation}
{\rm dist}(\cM_k,V_k)\leq \e_k \leq \e,\quad k=1,\dots,K.
\label{localacc}
\end{equation}
We may thus think of the family $(V_k)_{k=1,\dots,K}$ as a piecewise
affine approximation to $\cM$. We stress that, in contrast to the hierarchies 
$(V_n)_{n=0,\dots,m}$ produced by reduced modeling algorithms, the spaces $V_k$
do not have dimension $k$ and are not nested. {Most importantly,} $K$ is not limited
by $m$ while each $n_k$ is.

The objective of using a piecewise reduced model in the context of state
estimation is to have a joint control on the local accuracy $\e_k$
as expressed by \eqref{localacc} and on the stability of the PBDW when using any
individual $V_k$. This means that, for some prescribed 
$\mu>1$, we ask that
\begin{equation}
\mu_k=\mu(\bar V_k,W) \leq \mu, \quad k=1,\dots,K.
\label{localstab}
\end{equation}
According to \eqref{eq:pbdw-err-bound}, the worst case error bound over $\cM_k$ when using the PBDW method
with a space $V_k$ is given by the product $\mu_k\e_k$. This suggests to 
alternatively require from the collection $(V_k)_{k=1,\dots,K}$, that for some 
prescribed $\sigma>0$, one has
\begin{equation}
\sigma_k\coloneqq \mu_k\e_k \leq \sigma, \quad k=1,\dots,K.
\label{sigmaadm}
\end{equation}
{
This leads us to the following definitions.}

\begin{definition}
{The family $(V_k)_{k=1,\dots,K}$ is {\em $\sigma$-admissible}
if \eqref{sigmaadm} holds. It is {\em $(\e,\mu)$-admissible}
if \eqref{localacc} and \eqref{localstab} are jointly satisfied.}
\end{definition}

Obviously, any $(\e,\mu)$-admissible
family is $\sigma$-admissible with $\sigma\coloneqq \mu\e$.
In this sense the notion of $(\e,\mu)$-admissibility is thus more restrictive
than that of $\sigma$-admissibility. The benefit of the first notion is in 
the uniform control on the size of $\mu$ which is
critical in the presence of noise.

If $u\in \cM$ is our unknown state and $w=P_W u$ is its observation, we may apply the 
PBDW method for the different $V_k$ in the given family, which yields a corresponding 
family {of} estimators
\begin{equation}
u_k^*=u_k^*(w)={\rm argmin}\{\dist(v,V_k)\, : \, v\in \omega + W^\perp\}, \quad k=1,\dots,K.
\label{estimators}
\end{equation}
If $(V_k)_{k=1,\dots,K}$ is $\sigma$-admissible, we find that the accuracy bound
\begin{equation}
\|u-u_k^*\|\leq \mu_k{\rm dist}(u,V_k) \leq \mu_k\e_k= \sigma_k \leq \sigma,
\end{equation}
holds whenever $u\in \cM_k$.

Therefore, if in addition to the observed data $w$
one had an oracle giving the information on which portion $\cM_k$ of the manifold the unknown state sits,
we could derive an estimator with worst case error 
\begin{equation}
E_\wc\leq \sigma. 
\label{oracle}
\end{equation}
This information is, however, not available and such a worst case error estimate
cannot be hoped for, even with an additional multiplicative constant. Indeed, 
as we shall see below, $\sigma$ can be fixed arbitrarily small
by the user when building the family $(V_k)_{k=1,\dots,K}$, while we know from \eqref{eq:equiv-optimality} that the
worst case error is bounded from below by $E_{\wc}^*(\cM)\geq \frac 1 2 \delta_0$ which could be non-zero.
We will thus need to replace the ideal choice of $k$ by a model selection procedure
only based on the data $w$, that is, a map
\begin{equation}
w \mapsto k^*(w),
\end{equation}
leading to a choice of estimator $u^*=u^*_{k^*}=A_{k^*}$. We shall prove further that such an estimator
is able to achieve the accuracy
\begin{equation}
E_\wc(A_{k^*}, \cM) \leq \delta_{\sigma},
\end{equation}
that is, the benchmark introduced in \S 2.2. Before discussing this model selection, we discuss the 
existence and construction of $\sigma$-admissible or $(\e,\mu)$-admissible families.

\subsection{Constructing admissible reduced model families} \label{ssec:admissible_families}

For any arbitrary choice of $\e>0$ and $\mu\geq 1$, the existence of an $(\e,\mu)$-admissible family
results from the following observation: since the manifold $\cM$ is a compact set of $V$, there exists
a finite $\e$-cover of $\cM$, that is, a family $\bar u_1,\dots,\bar u_K \in V$ such that
\begin{equation}
\cM\subset \bigcup_{k=1}^K B(\bar u_k,\e),
\end{equation}
or equivalently, for all $v\in \cM$, there exists a $k$ such that $\|v-\bar u_k\|\leq \e$. With such an $\e$ cover,
we consider the family of trivial affine spaces defined by
\begin{equation}
V_k=\{\bar u_k\}=\bar u_k+\bar V_k, \quad \bar V_k=\{0\},
\end{equation}
thus with $n_k=0$ for all $k$. The covering property implies that \eqref{localacc} holds. On the other hand,
for the $0$ dimensional space, one has
\begin{equation}
\mu(\{0\},W)=1,
\end{equation}
and therefore \eqref{localstab} also holds. The family $(V_k)_{k=1,\dots,K}$ is therefore $(\e,\mu)$-admissible,
and also  $\sigma$-admissible with $\sigma=\e$.

This family is however not satisfactory for algorithmic purposes for two main reasons. First, the manifold
is not explicitly given to us and the construction of the centers $\bar u_k$ is by no means trivial. Second, 
asking for an $\e$-cover, would typically require that $K$ becomes extremely large as $\e$ goes to $0$. 
For example, assuming that the parameter to solution $y\mapsto u(y)$ has Lipschitz constant $L$,
\begin{equation}
\|u(y)-u(\tilde y)\|\leq L|y-\tilde y|, \quad y,\tilde y\in Y,
\end{equation}
for some norm $|\cdot|$ of $\bR^d$, then an $\e$ cover for $\cM$ would be induced by an $L^{-1}\e$ cover for $\rY$
which has cardinality $K$ growing like $\e^{-d}$ as $\e\to 0$. Having a family of moderate size $K$ is important
for the estimation procedure since we intend to apply the PBDW method for all $k=1,\dots,K$. 

In order to construct $(\e,\mu)$-admissible or $\sigma$-admissible families of better controlled size, we need to split the
manifold in a more economical manner than through an $\e$-cover, and use spaces $V_k$ of general
dimensions $n_k\in \{0,\dots,m\}$ for the various manifold portions $\cM_k$. To this end, we combine
standard constructions of linear reduced model spaces with an iterative splitting
procedure operating on the parameter domain $\rY$. Let us mention that various ways 
of splitting the parameter domain have already been considered in order to produce local 
reduced bases having both controlled cardinality and prescribed accuracy \cite{EPR2010,MS2013,BCDGJP2021}. However, these works are devoted to forward model reduction according to the terminology that we introduced in Section \ref{sec:intro-fwd-inv}. Here our goal is different since we want to control both the accuracy $\e$ and the stability $\mu$ with respect to 
the measurement space $W$.

We describe the greedy algorithm for constructing $\sigma$-admissible families, and explain how it
should be modified for $(\e,\mu)$-admissible families. For simplicity {we} consider the case where
$\rY$ is a rectangular domain with sides parallel to the main axes, the extension to a more general
bounded domain $\rY$ being done by embedding it in such a {hyper}-rectangle. We are given a prescribed target value $\sigma>0$
and the splitting procedure starts from $\rY$. 

At step $j$,
a disjoint partition of $\rY$ into rectangles $(\rY_k)_{k=1,\dots,K_j}$ with sides parallel
to the main axes has been generated. It induces
a partition of $\cM$ given by
\begin{equation}
\cM_k\coloneqq \{u(y)\,:\, y\in \rY_k\}, \quad k=1,\dots,K_j.
\end{equation}
To each $k\in \{1,\dots,K_j\}$ we associate a hierarchy of affine reduced basis spaces
\begin{equation} \label{eq:aff_from_basis}
V_{n,k}=\bar u_k+\bar V_{n,k}, \quad n=0,\dots, m.
\end{equation}
where $\bar u_k=u(\bar y_k)$ with $\bar y_k$ the vector defined as the center of the rectangle $\rY_k$. The nested linear 
spaces
\begin{equation}
\bar V_{0,k}\subset \bar V_{1,k}\subset \cdots\subset \bar V_{m,k}, \quad \dim({\bar V}_{n,k})=n,
\end{equation}
are meant to approximate the translated portion of the manifold $\cM_k-\bar u_k$. 
For example, they could be reduced basis spaces obtained by applying the 
greedy algorithm to $\cM_k-\bar u_k$, or spaces resulting from local $n$-term
polynomial approximations of $u(y)$ on the rectangle  $\rY_k$.
Each space $V_{n,k}$ has a given accuracy bound
and stability constant 
\begin{equation} \label{eq:acc_and_stab}
{\rm dist}(\cM_k,V_{n,k}) \leq  \e_{n,k}  \quad{\rm and}\quad \mu_{n,k}\coloneqq \mu(\bar V_{n,k},W).
\end{equation} 

We define the test quantity
\begin{equation} \label{eq:pm_criteria}
\tau_k=\min_{n=0,\dots,m} \mu_{n,k}\e_{n,k}.
\end{equation}
If $\tau_k\leq \sigma$, the rectangle $\rY_k$ is not {split} and becomes a member of the final partition.
The affine space associated to $\cM_k$ is 
\begin{equation}
V_k=\bar u_k+\bar V_k,
\end{equation}
where $V_k=V_{n,k}$ for the value of $n$ that minimizes $\mu_{n,k}\e_{n,k}$. The rectangles
$\rY_k$ with $\tau_k> \sigma$ are, on the other hand, {split} into a finite number of sub-rectangles 
in a way that we discuss below. This results in the new larger partition 
$(\rY_k)_{k=1,\dots,K_{j+1}}$ after relabelling the $\rY_k$. The algorithm terminates
at the step $j$ as soon as $\tau_k\leq \sigma$ for all $k=1,\dots,K_j=K$,
and the family $(V_k)_{k=1,\dots,K}$ is $\sigma$-admissible.
In order to obtain an $(\e,\mu)$-admissible family, we simply modify the test quantity $\tau_k$
by defining it instead as
\begin{equation}
\tau_k\coloneqq \min_{n=0,\dots,m}\max\Big\{\frac {\mu_{n,k}}\mu, \frac {\e_{n,k}}\e\Big\}
\end{equation}
and splitting the cells for which $\tau_k>1$.

The splitting of one single rectangle $\rY_k$ can be performed in various ways. When the parameter 
dimension $d$ is moderate, we may subdivide each side-length at the mid-point, resulting into $2^d$ sub-rectangles
of equal size. This splitting becomes too   costly as $d$ gets large, in which case it is preferable to make a choice 
of $i\in \{1,\dots,d\}$ and subdivide $\rY_k$ at the mid-point of the side-length in the $i$-coordinate, resulting
in  only $2$ sub-rectangles. In order to decide which coordinate to pick, we consider the $d$ possibilities and
take the value of $i$ that minimizes the quantity
\begin{equation}
\tau_{k,i}=\max \{\tau_{k,i}^{{-}},\tau_{k,i}^+\},
\end{equation}
where $(\tau_{k,i}^{{-}},\tau_{k,i}^+)$ are the values of $\tau_k$ for the two subrectangles obtained by splitting along the $i$-coordinate.
In other words, we split in the direction that decreases $\tau_k$  most effectively. In order to be certain that all side-length
are eventually {split}, we can mitigate the greedy choice of $i$ in the following way: if $\rY_k$ has been generated
by $l$ consecutive refinements, and therefore has volume $|\rY_k|=2^{-l}|Y|$, and if $l$ is even, we choose $i={(l/2 \,{\rm mod}\,d)}$. 
This 
means that at each even level we split in a cyclic manner in the coordinates $i\in\{1,\dots,d\}$.

Using such elementary splitting rules, we are ensured that the algorithm must terminate. Indeed, we are guaranteed
that for any $\eta>0$, there exists a level $l=l(\eta)$ such that any rectangle $\rY_k$ generated by $l$ consecutive refinements
has side-length smaller than $2\eta$ in each direction. Since the parameter-to-solution map is assumed to be continuous, for any $\e>0$, we can pick
$\eta>0$ such that
\begin{equation}
\|y-\tilde y\|_{\ell^\infty}\leq \eta \implies \|u(y)-u(\tilde y)\|\leq \e, \quad y,\tilde y\in Y.
\end{equation}
Applying this to $y\in \rY_k$ and $\tilde y=\bar y_k$, we find that {for $\bar u_k = u(\bar y_k)$}
\begin{equation}
\|u-\bar u_k\| \leq \e, \quad u\in \cM_k.
\end{equation}
Therefore, for any rectangle $\rY_k$ of generation $l$, we find that the trivial affine space $V_k=\bar u_k$ has local accuracy $\e_k\leq \e$
and $\mu_k=\mu(\{0\},W)=1\leq \mu$, which implies that such a rectangle would not anymore be refined by the algorithm.

\subsection{Reduced model selection and recovery bounds} \label{ssec:recovery_bounds}
\label{sec:model-selection}
We return to the problem of selecting an estimator 
within the family $(u_k^*)_{k=1,\dots,K}$ defined by \eqref{estimators}.
In an idealized version, the selection procedure picks the value $k^*$
that minimizes the distance of $u_k^*$ to the solution manifold, that is,
\begin{equation}
k^*={\rm argmin}\{ \dist(u_k^*,\cM)\,:\, k=1,\dots,K\}
\label{idealsel}
\end{equation}
and takes for the final estimator
\begin{equation}
\label{ustar}
u^*=u^*(w)\coloneqq A_{k^*}(w)=u^*_{k^*}(w).
\end{equation}
Note that $k^*$ also depends on the observed data $w$.
This estimation procedure is not realistic since the computation of the distance
of a known function $v$ to the manifold
\begin{equation}
\dist(v,\cM)=\min_{y\in Y} \|u(y)-v\|,
\end{equation}
is a high-dimensional non-convex  problem which
necessitates to explore the whole solution manifold. A more realistic 
procedure is based on replacing this distance by a
surrogate quantity 
${\cal S}(v,\cM)$ that is easily computable
and satisfies a uniform equivalence
\begin{equation}
r\dist(v,\cM)\leq {\cal S}(v,\cM)\leq R \dist(v,\cM), \quad v\in V,
\label{surframe}
\end{equation}
for some constants $0<r\leq R$. We then instead take for $k^*$ the value
that minimizes this surrogate, that is,
\begin{equation}
k^*={\rm argmin}\{ {\cal S}(u_k^*,\cM)\,:\, k=1,\dots,K\}.
\label{realisticsel}
\end{equation}
Before discussing the derivation of ${\cal S}(v,\cM)$ in concrete cases, we establish a recovery
bound in the absence of model bias and noise.

\begin{theorem}
\label{theorecovery}
Assume that the family $(V_k)_{k=1,\dots,K}$ is $\sigma$-admissible for some $\sigma>0$. Then, the idealized estimator
based on \eqref{idealsel}, \eqref{ustar}, satisfies the worst case error estimate
\begin{equation}
E_\wc(A_{k^*}, \cM)=\max_{u\in \cM} \|u-u^*(P_Wu)\| \leq \delta_\sigma,
\label{recovideal}
\end{equation}
where $\delta_\sigma$ is the benchmark quantity defined in \eqref{benchapp}. When using the estimator
based on \eqref{realisticsel}, the worst case error estimate is modified into
\begin{equation}
E_\wc(A_{k^*}, \cM) \leq \delta_{\kappa\sigma}, \quad \kappa=\frac R r>1.
\label{recovreal}
\end{equation}
\end{theorem}

In the above result, we do not obtain the best possible accuracy 
satisfied by the different $u_k^*$, since we do not have
an oracle providing the information on the best choice of $k$. We can show that this order of 
accuracy is attained in the particular case where the measurement map 
$P_W$ is injective on $\cM$ (which implies $\delta_0=0$).

\begin{theorem}
\label{thm:3.3}
Assume that $\delta_0=0$ and that
\begin{equation}
\mu(\cM,W)=\frac 1 2\sup_{\sigma>0} \frac{\delta_\sigma}{\sigma}<\infty.
\end{equation}
Then, for any given state $u\in \cM$ with observation $w=P_Wu$,
the estimator $u^*$ obtained by the model selection procedure \eqref{realisticsel} satisfies
the oracle bound
\begin{equation}
\|u-u^*\| \leq C\min_{k=1,\dots,K} \|u-u_k^*\|, \quad C\coloneqq 2\mu(\cM,W)\kappa.
\label{oracle1}
\end{equation} 
In particular, if $(V_k)_{k=1,\dots,K}$ is $\sigma$-admissible, it satisfies
\begin{equation}
\|u-u^*\| \leq C\sigma.
\label{oracle2}
\end{equation}
\end{theorem}

The next theorem outlines how to incorporate model bias and noise in the recovery bound, provided that we have a control on the stability of the PBDW method, through a uniform bound on $\mu_k$, which holds when we use $(\e,\mu)$-admissible families.

\begin{theorem}
\label{thm:3.4}
Assume that the family $(V_k)_{k=1,\dots,K}$ is $(\e,\mu)$-admissible for some $\e>0$ and $\mu\geq 1$. 
If the observation is $w=P_Wu+\eta$ with $\|\eta\|\leq \e_{noise}$,
and if the true state does not lie in $\cM$ but satisfies ${\rm dist}(u,\cM)\leq \e_{model}$,
then, the estimator based on \eqref{realisticsel} satisfies the estimate
\begin{equation}
\|u-u^*(w)\|\leq \delta_{\kappa\rho}+\e_{noise}, \quad \rho\coloneqq \mu(\e+\e_{noise})+(\mu+1)\e_{model},\quad \kappa=\frac R r,
\label{pertest}
\end{equation}
and the idealized estimator based on \eqref{idealsel} satifies a similar estimate with $\kappa=1$.
\end{theorem}

\section{Bibliographical Remarks/Connections with other works}
\label{sec:biblio}

\subsection{A bit of history on the use of reduced models to solve inverse problems}

We often think of reduced order models only as a vehicle to speed up calculations in forward reduced modeling tasks according to the terminology that we introduced in Section \ref{sec:optim-benchmarks}. However, reduced order models $V_n$ play also a very prominent role in the inverse problem approach that we have presented. They are the main vehicle for building implementable reconstruction algorithms whose performance can be proven to be close to optimal.

In fact, the idea of using reduced models to solve inverse problems has actually a relatively long history. It can be traced back at least to the gappy POD method, first introduced in \cite{ES1995} by Everson and Sirovich. There, the authors address the problem of restoring a full image from partial pixel observations by using a least squares strategy involving a reconstruction on linear spaces obtained by PCA. The same strategy was then brought to other fields such as fluid and structural applications,
see \cite{Willcox2006}. The introduction of a reduced model can be seen as an improvement with respect to working with one single background function as is done in methods such as 3D-VAR, see \cite{Lorenc1981, Lorenc1986}. In contrast to the present work and the PBDW method in general, the gappy POD method is formulated on the euclidean space $V= \bR^\cN$, with $\cN\in \bN$ typically much larger than $m$ and $n$. It uses linear reduced models $V_n$ obtained by PCA and measurement observations are typically point-wise vector entries, that is $\omega_i = e_i$ with $e_i\in \bR^\cN$ being the $i$-th unit vector. For that particular choice of ambient space and reduced models, the linear PBDW method is very close to gappy POD. It is however not entirely equivalent since PBDW presents a certain component in $W\cap V_n^\perp$ which is missing in gappy POD. For the case of a general Hilbert space, there is a connection between the linear PBDW is equivalent to the Generalized Empirical Interpolation Method as we have outlined in Section \ref{sec:geim-algo}.

It is also interesting to note that the linear PBDW reconstruction algorithm \eqref{eq:pbdw-lin} was proposed simultaneously in the field of model order reduction and by researchers seeking to build infinite dimensional generalizations of compressed sensing (see \cite{AHP2013}). In the applications of this community, $V_n$ is usually chosen to be a ``multi-purpose'' basis such as the Fourier basis, as opposed to our current envisaged applications in which $V_n$ is a subspace specifically tailored to approximate $\cM$. However, the results that we have sumarized here are general, and they remain valid also for these types of ``multi-purpose'' subspaces.



In the above landscape of methods, the piecewise affine extension of PBDW of Section \ref{sec:piecewise-affine} can be interpreted as a further generalization step which comes with optimal reconstruction guarantees. The strategy is based on an offline partitioning of the manifold $\cM$ in which, for each element of the partition, we compute reduced models. We then decide with a data-driven approach which reduced model is the most appropriate for the reconstruction. The idea of partitioning the manifold and working with different reduced order models for each partition is new for the purpose of addressing inverse problems. It has however been explored in works that focus on the forward modeling problem see, e.g., \cite{AZF2012, PBWB2014, Karlberg2015,AH2016}. For forward modeling, the piece-wise strategy enters into the general topic of nonlinear forward model reduction for which little is known in terms of the performance guarantees. A first step towards a cohesive theory for nonlinear forward model reduction has recently been proposed in \cite{BCDGJP2021}, in relation with the general concept of library widths \cite{Temlyakov1998}.

\subsection{For further reading}
\begin{itemize}
\item \textbf{Noise and physical model error:} For the readers interested in further aspects connected to noise, we refer to \cite{EF2021} for a study on optimal benchmarks with noise. Some algorithms that attempt to do some denoising have been presented in \cite{MPPY2015_, Taddei2017, ABGMM2017, GMMT2019}. A contribution that aims to learn physical model corrections can be found in \cite{AGV2019}.
\item \textbf{Beyond the Hilbertian framework:} The general framework of optimal recovery that we have introduced in Section \ref{sec:optim-benchmarks} can be extended to general Banach spaces as has been done in \cite{DPW2017}.
\item \textbf{GEIM and variants:} The GEIM can also be formulated in general Banach spaces (see \cite{MMT2016}). This justifies why GEIM is a generalization of the celebrated EIM originally introduced in \cite{BMNP2004} (see also, e.g., \cite{GMNP2007}): if we work with a manifold in the Banach space of continuous functions $V=\cC(\Omega)$ with the sup-norm
$$
\Vert v \Vert_{\infty} \coloneqq \sup_{x\in \Omega} | v(x) |,\quad \forall v \in \cC(\Omega),
$$
GEIM boils down to EIM when we use the dictionary composed of pointwise evaluations
$$
\cD = \{ \delta_x \cond x\in \Omega \}.
$$
EIM and GEIM strongly interweave forward and inverse problems since the exact same algorithm can be applied for both purposes. EIM was originally introduced to address forward model reduction of nonlinear PDEs. It can also be applied as a reconstruction algorithm as outlined in Section \ref{sec:geim-algo}, and GEIM allows to apply it in basically any functional setting.
\item \textbf{Applications:} Among the applicative problems that have been addressed with the present inverse problem approach, we can cite:
\begin{itemize}
\item Acoustics problems: \cite{MPPY2015}.
\item Biomedical problems: \cite{GLM2021, GGLM2021, GLM2021-geom}.
\item Air quality: \cite{HCBM2019}.
\item Nuclear engineering: \cite{ABGMM2017, ABCGMM2018}.
\item Welding: \cite{PKRR2021}.
\end{itemize}
\end{itemize}

\appendix
\section*{Appendix}

\section{Practical computation of $A_n$, the linear PBDW algorithm}
\label{appendix:linear-pbdw}
Let $X$ and $Y$ be two finite dimensional subspaces of $V$ and let
\begin{align*}
P_{X | Y} : Y &\to X \\
y &\mapsto P_{X | Y} (y)
\end{align*}
be the orthogonal projection into $X$ restricted to $Y$. That is, for any $y\in Y$, $P_{X | Y}(y)$ is the unique element $x\in X$ such that 
$$
\left< y - x, \tilde x\right>= 0,\quad  \forall \tilde x \in X.
$$

\begin{lemma}
Let $\Wm$ and $\Vn$ be an observation space and a reduced basis of dimension $n\leq m$ such that $\beta(\Vn, \Wm)>0$. Then the linear PBDW algorithm defined in \eqref{eq:pbdw-lin} is given by
\begin{equation}
\label{eq:explicit-u}
A_n(\omega) = \omega + v^*_{m,n} - P_W v^*_{m,n},
\end{equation}
with
\begin{equation}
\label{eq:explicit-v}
v^*_{m,n} = \( P_{V_n | \Wm} P_{\Wm | V_n} \)^{-1} P_{V_n | \Wm} (\omega).
\end{equation}
\end{lemma}
\begin{proof}
By formula \eqref{eq:pbdw-lin}, $A_n(\omega)$ is a minimizer of
\begin{align}
\min_{u \in \omega + \Wm^\perp } \dist(u, V_n)^2
&= \min_{u \in \omega + \Wm^\perp } \min_{v\in V_n} \Vert u - v\Vert^2 \\
&= \min_{v\in V_n}  \min_{\eta \in \Wm^\perp} \Vert \omega + \eta - v \Vert^2 \\
&= \min_{v\in V_n} \Vert \omega -v - P_{\Wm^\perp}(\omega -v) \Vert^2 \\
&= \min_{v\in V_n} \Vert \omega -v + P_{\Wm^\perp}(v) \Vert^2 \\
&= \min_{v\in V_n} \Vert \omega - P_{\Wm}(v) \Vert^2. \label{eq:vn}
\end{align}
The last minimization problem is a classical least squares optimization. Any minimizer $v^*_{m,n}\in V_n$ satisfies the normal equations
$$
P^*_{\Wm|V_n}  P_{\Wm|V_n} v^*_{m,n} = P^*_{\Wm|V_n} \omega,
$$
where $P^*_{\Wm|V_n} : V_n \to \Wm$ is the adjoint operator of $P_{\Wm|V_n}$. Note that $P^*_{\Wm|V_n}$ is well defined since $\beta(V_n, \Wm)= \min_{v\in V_n} \Vert P_{\Wm|V_n} v \Vert / \Vert v \Vert >0 $, which implies that $P_{\Wm|V_n}$ is injective and thus admits an adjoint. Furthermore, since for any $\omega \in \Wm$ and $v\in V_n$, $\< v, \omega \>=\< P_{\Wm|V_n} v , \omega \> = \< v , P_{V_n|\Wm} \omega \>$, it follows that $P^*_{\Wm|V_n} = P_{V_n|\Wm}$, which finally yields that the unique solution of the least squares problem is
$$
v^*_{m,n} = \( P_{V_n | \Wm} P_{\Wm | V_n} \)^{-1} P_{V_n | \Wm} \omega .
$$
Therefore $A_n(\omega) = \omega + \eta^*_{m,n} = \omega + v^*_{m,n} - P_\Wm v^*_{m,n}$.
\end{proof}

\textbf{Algebraic formulation:} The explicit expression \eqref{eq:explicit-v} for $v^*_n$ allows to easily derive its algebraic formulation. Let $F$ and $H$ be two finite-dimensional subspaces of $V$ of dimensions $n$ and $m$ respectively in the Hilbert space $V$ and let $\cF=\{f_i\}_{i=1}^n$ and $\cH=\{h_i\}_{i=1}^m$ be a basis for each subspace respectively. The Gram matrix associated to $\cF$ and $\cH$ is
$$
\bG(\cF, \cH) = \left(  \left< f_i, h_j\right> \right)_{\substack{1\leq i \leq n \\ 1\leq j \leq m}}.
$$
These matrices are useful to express the orthogonal projection
$P_{ F | H}: H\mapsto F$ in the bases $\cF$ and $\cH$ in terms of the matrix
\begin{equation}
\label{eq:proj-matrix}
\bP_{F | H} = \bG(\cF, \cF)^{-1} \bG(\cF, \cH).
\end{equation}
As a consequence, if $\cV_n = \{ v_i \}_{i=1}^n$ is a basis of the space $V_n$ and $\cW_m = \{\omega_i\}_{i=1}^m$ is the basis of $W_m$ formed by the Riesz representers of the linear functionals $\{\ell_i\}_{i=1}^m$, the coefficients $\textbf{v}^*_{m,n}$ of the function $v^*_{m,n}$ in the basis $\cV_n$ are the solution to the normal equations
\begin{equation}
\label{eq:normal-eqs1}
\bP_{V_n | W_m} \bP_{W_m | V_n} 
\textbf{v}^*_{m,n} =  \bP_{V_n | W_m} \bG(\cW_m, \cW_m)^{-1} \textbf{w},
\end{equation}
where $\textbf{w}$ is the vector of measurement observations
$$
\textbf{w} = (\left< u, \omega_i\right>)_{i=1}^m,
$$
and from formula \eqref{eq:proj-matrix},
\begin{equation}
\begin{cases}
\bP_{V_n | W_m} &= \bG(\cV_n, \cV_n)^{-1} \bG(\cV_n, \cW_m), \\
\bP_{W_m | V_n} &= \bG(\cW_m, \cW_m)^{-1} \bG(\cW_m, \cV_n).
\end{cases}
\end{equation}
Usually $\textbf{v}^*_{m,n}$ is computed with a QR decomposition or any other suitable method. Once $\textbf{v}^*_{m,n}$ is found, the vector of coefficients $\textbf{u}_{m,n}^*$ of $A_n(\omega)$ easily follows.

\section{Practical computation of $\beta(V_n, W_m)$}
\label{appendix:beta}
Let $V_n$ and $W_m$ be two linear subspaces of $V$ of dimensions $n$ and $m$ respectively, and with $n\leq m$. The inf-sup constant between these spaces was defined in equation \eqref{eq:beta}, and we recall it here:
\begin{equation}
\label{eq:beta2}
\beta_n = \beta(V_n, W_m) \coloneqq
\min_{v\in V_n} \max_{w\in W_m} \frac{\<v,w\>}{\|v\| \,\|w\|} 
= \min_{v\in V_n} \frac{\Vert P_{W_m} v \Vert}{\Vert v \Vert}.
\end{equation}
The last equality comes from the fact that
$$
 \max_{w\in W_m} \frac{\<v,w\>}{\|w\|}
 =  \max_{w\in W_m} \frac{\<P_{W_m}v,w\>}{\|w\|}
 = \Vert P_{W_m} v \Vert,
 \quad \forall v \in V_n.
$$
Let $\cV_n = \{ v_i \}_{i=1}^n$ be a basis of the space $V_n$ and let $\textbf{c}$ be the coefficients of an element $v\in V_n$ in the basis $\cV_n$. For any nonzero $v\in V_n$, we can thus write
\begin{equation}
\label{eq:beta-eigenvalue}
\beta_n = \min_{v\in V_n}
\frac{\Vert P_{W_m} v \Vert_V^2}{\Vert v \Vert_V^2}
=
\min_{\textbf{c}\in \bR^n}
\frac{\textbf{c}^T \bM(\cV_n, \cW_m) \textbf{c}}{\textbf{c}^T \bG(\cV_n, \cV_n) \textbf{c}}
\end{equation}
where
$$
\bM(\cV_n, \cW_m) \coloneqq \left( \left< P_{W_m} v_i, P_{W_m} v_j \right> \right)_{1\leq i, j \leq n}
$$
is a symmetric matrix.

Let us make a few remarks before giving an implementable expression for $\bM(\cV_n, \cW_m)$.
First, note that the value of $\beta_n$ does not depend on the selected bases $\cV_n$ and $\cW_m$. For example, using a basis $\widetilde \cV_n$ instead of $\cV_n$ amounts to  changing the variable $\textbf{c}$ by $\widetilde{\textbf{c}} = \bU \textbf{c}$ for an invertible matrix $\bU$, and this does not affect the value of the minimizer. Second, note that formula \eqref{eq:beta-eigenvalue} shows that $\beta_n$ is the smallest eigenvalue of the generalized eigenvalue problem
$$
\text{find } (\lambda, \textbf{c}) \in \bR\times \bR^n-\{0\} \quad \text{s.t.} \quad \bM(\cV_n, \cW_m) \textbf{c} = \lambda \bG(\cV_n, \cV_n) \textbf{c}.
$$
Since $\bG(\cV_n, \cV_n)$ and $\bM(\cV_n, \cW_m)$ are symmetric, positive definitive, the eigenvalues $\lambda$ are positive, and having $\beta_n >0$ is equivalent to the invertibility of $\bM(\cV_n, \cW_m)$. We can transform the generalized eigenvalue problem in a classical eigenvalue problem by multiplying by the inverse of $\bG(\cV_n, \cV_n)$. Also, remark that we have important simplifications when $\cV$ and/or $\cW_m$ are orthonomal bases since in that case $\bG(\cV_n, \cV_n) $ and $\bG(\cW_m, \cW_m)$ become the identity matrices.

We next give an explicit expression for $\bM(\cV_n, \cW_m)$. Since the coordinates in $\cV_n$ of the $i$-th basis function $v_i$ are given by the $i$-th canonical vector $\textbf{e}_i \in \bR^n$, using formula \eqref{eq:proj-matrix} we deduce that the coordinates of $P_{W_m} v_i$ in $\cW_m$ are given by
$$
\textbf{p}_i \coloneqq \bP_{W_m | V_n} \textbf{e}_i = \bG(\cW_m, \cW_m)^{-1} \bG(\cW_m, \cV_n) \textbf{e}_i,\quad \forall i \in \{1,\dots, n\}.
$$
Therefore
\begin{align}
\left< P_{W_m} v_i, P_{W_m} v_j \right>_V
&= \textbf{p}_i^T \bG(\cW_m, \cW_m) \textbf{p}_j \\
&= \textbf{e}^T_i \bG^T(\cW_m, \cV_n)  \bG^{-1}(\cW_m, \cW_m) \bG(\cW_m, \cV_n) \textbf{e}_j
,\quad \forall (i,j) \in \{1,\dots, n\}^2,
\end{align}
and
$$
\bM(\cV_n, \cW_m) = \bG^T(\cW_m, \cV_n)  \bG^{-1}(\cW_m, \cW_m) \bG(\cW_m, \cV_n).
$$

\bibliographystyle{unsrt}
\bibliography{/home/olga/Documents/research/literature/literature}
\end{document}